\documentclass[12pt, leqno]{amsart}

\usepackage{lscape}
\usepackage{graphicx}
\usepackage{amssymb,amsmath,amsthm,amscd}
\usepackage{mathrsfs}
\usepackage{enumerate}
\usepackage{verbatim}
\usepackage[usenames,dvipsnames]{color}
\usepackage[colorlinks=true, pdfstartview=FitV,  linkcolor=blue,citecolor=blue,urlcolor=blue, bookmarks=false]{hyperref}
\usepackage[all]{xy}
\usepackage{oldgerm}

\usepackage[normalem]{ulem}  % for strikeout \sout
\usepackage{xspace}

\newcommand{\nc}{\newcommand}
\numberwithin{equation}{section}

\newenvironment{blue}{\relax\color{blue}}{\hspace*{.5ex}\relax}

\newcommand{\bj}{\begin{jaune}}
\newcommand{\ej}{\end{jaune}}

\newcommand{\bebl}{\begin{blue}}
\newcommand{\ebl}{\end{blue}}

\theoremstyle{plain}
\newtheorem*{theorem*}{Theorem}
\newtheorem{lemma}{Lemma}[section]
\newtheorem{prop}[lemma]{Proposition}
\newtheorem{theorem}[lemma]{Theorem}

\newcommand{\Prop}{\begin{prop}}
\newcommand{\enprop}{\end{prop}}
\newcommand{\Lemma}{\begin{lemma}}
\newcommand{\enlemma}{\end{lemma}}
\newcommand{\Th}{\begin{theorem}}
\newcommand{\enth}{\end{theorem}}
\newtheorem{corollary}[lemma]{Corollary}
\newcommand{\Cor}{\begin{corollary}}
\newcommand{\encor}{\end{corollary}}
\newtheorem{definition}[lemma]{Definition}
\newtheorem{conjecture}[lemma]{Conjecture}
\newcommand{\Def}{\begin{definition}}
\newcommand{\edf}{\end{definition}}
\newtheorem{sublemma}[lemma]{Sublemma}
\newcommand{\Sub}{\begin{sublemma}}
\newcommand{\ensub}{\end{sublemma}}

\theoremstyle{definition}
\newtheorem{remark}[lemma]{Remark}

\newtheorem{Convention}[lemma]{Convention}

\newcommand{\Conv}{\begin{Convention}}
\newcommand{\enconv}{\end{Convention}}
\nc{\Con}{\begin{conjecture}}
\nc{\encon}{\end{conjecture}}
\nc{\Rem}{\begin{remark}}
\nc{\enrem}{\end{remark}}

\newcommand{\Q}{\mathbb {Q}}

\newcommand{\Z}{{\mathbb Z}}
\newcommand{\B}{{\mathbf{B}}}

\newcommand{\one}{{\bf{1}}}
\newcommand{\seteq}{\mathbin{:=}}

\newcommand{\hd}{{\operatorname{hd}}}
\newcommand{\soc}{{\operatorname{soc}}}

\newcommand{\g}{\mathfrak{g}}
\newcommand{\h}{\mathfrak{h}}
\newcommand{\n}{\mathfrak{n}}

\newcommand{\Uq}[1][{\mathfrak{g}}]{{U_q(#1)}}

\newcommand{\Hom}{\operatorname{Hom}}
\newcommand{\HOM}{\operatorname{H\textsc{om}}}

\nc{\N}{\Z_{\ge0}}

\newenvironment{myequation}
{\relax\setlength{\arraycolsep}{1pt}\begin{eqnarray}}
{\end{eqnarray}}
\newenvironment{myequationn}
{\relax\setlength{\arraycolsep}{1pt}\begin{eqnarray*}}
{\end{eqnarray*}}

\nc{\eq}{\begin{myequation}}
\nc{\eneq}{\end{myequation}}
\nc{\eqn}{\begin{myequationn}}
\nc{\eneqn}{\end{myequationn}}

\newcommand{\hs}{\hspace*}

\newcommand{\To}[1][{\hs{2ex}}]{\xrightarrow{\,#1\,}}

\newcommand{\on}{\operatorname}
\newcommand{\Ker}{\on{Ker}}

\newcommand{\bni}{\be[{\rm(i)}]}
\newcommand{\bna}{\be[{\rm(a)}]}

\newcommand{\QED}{\end{proof}}
\newcommand{\Proof}{\begin{proof}}

\newcommand{\soplus}{\mathop{\mbox{\normalsize$\bigoplus$}}\limits}

\newcommand{\id}{\on{id}}
\newcommand{\ba}{\begin{array}}
\newcommand{\ea}{\end{array}}

\newcommand{\monoto}{\rightarrowtail}

\newcommand{\set}[2]{\left\{#1 \mid #2 \right\}}

\newcommand{\eqsub}{\begin{subequations}\begin{eqnarray}}
\newcommand{\eneqsub}{\end{eqnarray}\end{subequations}}

\newcommand{\ol}{\overline}

\nc{\la}{\lambda}
\nc{\lam}{\lambda}
\nc{\U}[1][\g]{U_q(#1)}
\nc{\te}{\tilde{e}}
\nc{\tei}{\tilde{e}_i}
\nc{\tf}{\tilde{f}}
\nc{\tfi}{\tilde{f}_i}
\nc{\tU}{\widetilde U_q(\g)}
\nc{\tE}{\tilde{E}}
\nc{\tF}{\widetilde{F}}
\nc{\tK}{\widetilde{K}}

\nc{\tk}{\tilde{k}}
\nc{\tkone}{\tk_{\ol{1}}}
\nc{\teone}{\tilde{e}_{\ol{1}}}
\nc{\tfone}{\tilde{f}_{\ol{1}}}

\nc{\teibar}{\tilde{e}_{\ol{i}}} \nc{\tfibar}{\tilde{f}_{\ol{i}}}
\nc{\tki}{{\tk}_{\ol {i}}}

\nc{\BZ}{{\mathbb{Z}}}
\nc{\al}{\alpha}
\nc{\qs}{{q}}
\nc{\lan}{\langle}
\nc{\ran}{\rangle}
\nc{\re}{{\mathrm{re}}}
\nc{\wt}{\operatorname{wt}}
\nc{\ch}{\operatorname{ch}}
\nc{\Um}[1][\g]{U^-_q(#1)}
\nc{\Ue}{U^+_q(\g)}
\nc{\eps}{\varepsilon}
\nc{\vphi}{\varphi}
\nc{\sphi}{\varphi^*}
\nc{\seps}{\varepsilon^*}

\nc{\nn}{\nonumber}
\def\max{{\mathop{\mathrm{max}}}}
\nc{\vph}{\varphi}
\nc{\cls}{{\operatorname{cl}}}
\nc{\Wt}{{\operatorname{Wt}}}
\nc{\Us}{U'_q(\g)}
\nc{\La}{\Lambda}
\nc{\tLa}{\widetilde\Lambda}
\nc{\ro}{{\rm(}}
\nc{\rf}{{\rm)}\xspace}
\nc{\norm}{{\mathrm{norm}}}
\nc{\qbox}{\quad\mbox}
\nc{\braid}{{\mathfrak{B}}}
\nc{\Ad}{\operatorname{Ad}}
\nc{\Aut}{\operatorname{Aut}}
\nc{\dt}[1]{\tilde{\tilde #1}}
\nc{\Sn}{S^{{\mathrm{norm}}}}
\nc{\aff}{{\rm{aff}}}
\nc{\rk}{{\mathrm{rk}}}
\nc{\tP}{\widetilde{P}}
\nc{\tW}{\widetilde{W}}
\nc{\Dyn}{\mathrm{Dyn}}
\nc{\tD}{\widetilde{\Delta}}
\nc{\height}[1]{{\operatorname{ht}}(#1)}
\nc{\bl}{\bigl(}
\nc{\br}{\bigr)}
\nc{\Hecke}{\mathrm{H}}
\nc{\HA}{\Hecke^{\mathrm{A}}}
\nc{\HB}{\Hecke^{\mathrm{B}}}
\newcommand{\scbul}{{\,\raise1pt\hbox{$\scriptscriptstyle\bullet$}\,}}
\nc{\vac}{{\phi}}
\nc{\Bt}{\B_\theta(\g)}
\nc{\be}{\begin{enumerate}}
\nc{\ee}{\end{enumerate}}
\nc{\low}{{\mathrm{low}}}
\nc{\upper}{{\mathrm{up}}}
\nc{\Zodd}{\Z_{\mathrm{odd}}}
\nc{\Ft}[1][n]{\mathbb{P}\mathrm{ol}_{#1}}
\nc{\Ftf}[1][n]{\widetilde{\mathbb{P}\mathrm{ol}}_{#1}}
\nc{\KA}{\on{K}^{\mathrm{A}}}
\nc{\KB}{\on{K}^{\mathrm{B}}}
\nc{\Res}{\on{Res}}
\nc{\Fc}[1][{n,m}]{\mathbf{F}_{#1}}
\nc{\tphi}{\tilde{\varphi}}
\nc{\CO}{\mathscr{O}}
\nc{\inte}{\mathrm{int}}

\nc{\vs}{\vspace*}
\nc{\tLt}{\widetilde{L}}
\nc{\tL}{\widetilde{\Lambda}}
\nc{\tu}{\tilde{u}}
\nc{\noi}{\noindent}
\nc{\heigh}{\mathfrak{t}}
\nc{\lowest}{\mathfrak{l}}
\nc{\rootl}{\mathsf{Q}}
\nc{\cl}{\colon}
\nc{\uqpg}{U'_q(\mathfrak g)}
\nc{\uq}{\uqpg}
\nc{\Oh}{\widehat{\mathcal{O}}}

\nc{\KLR}{KLR algebra}
\nc{\KLRs}{KLR algebras}
\nc{\cor}{\mathbf{k}}
\nc{\cora}{{\cor(A)}}
\nc{\haut}{\mathrm{ht}}
\nc{\tens}{\mathop\otimes}
\nc{\gmod}{\text{-$\mathrm{gmod}$}}
\nc{\gMod}{\text{-$\mathrm{gMod}$}}
\nc{\proj}{\text{-$\mathrm{proj}$}}
\nc{\gproj}{\text{-$\mathrm{gproj}$}}
\nc{\smod}{\text{-$\mathrm{mod}$}}
\nc{\Mod}{\text{-$\mathrm{Mod}$}}
\nc{\Rnorm}{R^{\mathrm{norm}}}

\nc{\Vhat}{\widehat{V}}
\nc{\F}{\mathcal{F}}

\def\T{{\mathcal T}}

\nc{\fd}[1][A]{\on{\mathrm{flat.dim}_{#1}}}
\nc{\bP}{{\mathbb{P}}}
\nc{\bPh}{\widehat{\mathbb{P}}}
\nc{\bK}[1][{n}]{\widehat{\mathbb{K}}_{#1}}
\nc{\bV}[1][{n}]{\widehat{V}^{\otimes{#1}}}
\nc{\bVK}[1][{n}]{\widehat{V}^{\otimes{#1}}_{\widehat{\mathbb{K}}}}
\nc{\hV}{\widehat{V}}
\nc{\opp}{\mathrm{opp}}
\nc{\col}{\colon}
\nc{\bnum}{\be[{\rm(i)}]}
\nc{\bnam}{\be[{\rm(a)}]}
\nc{\oep}{\epsilon}
\nc{\qtext}[1][{and}]{\quad\text{#1}\quad}

\nc{\qtextq}[1]{\quad\text{#1}\quad}
\nc{\longtwoheadrightarrow}[1][]{\xymatrix{\ar@{->>}[r]^-{{#1}}&}}
\nc{\epiTo}[1][]{\longtwoheadrightarrow[{#1}]}
\nc{\epito}{\twoheadrightarrow}
\nc{\monoTo}[1][]{\xymatrix{\ar@{>->}[r]^-{{#1}}&}}
\nc{\sym}{\mathfrak{S}}
\nc{\inp}[1]{{({#1})_{\mathrm{n}}}}
\nc{\rtl}{\rootl}
\nc{\wdt}{\widetilde}
\nc{\etens}{\boxtimes}
\nc{\ds}[1]{\mathrm{d}(#1)}
\nc{\rmat}[1]{{\mathbf{r}}_%
{\mspace{-2mu}\raisebox{-.6ex}{${\scriptstyle{#1}}$}}}
\nc{\rmats}[1]{{\mathbf{r}}_%
{\mspace{-2mu}\raisebox{-.6ex}{${\scriptscriptstyle{#1}}$}}}
\nc{\shc}{\mathcal{C}}
\nc{\shs}{\mathscr{S}}
\nc{\Fct}{{\on{Fct}}}
\nc{\tC}{\widetilde{\shc}}
\nc{\Zp}{\Z_{\ge0}}
\nc{\tPhi}{\widetilde{\Phi}}
\nc{\tT}{{\widetilde{\T}}}
\nc{\Ob}{\on{Ob}}
\nc{\bwr}{\mbox{\large$\wr$}}
\nc{\Img}{\on{Im}}
\nc{\Ab}{\mathcal{A}^{\mathrm{big}}}
\nc{\Sb}{\mathcal{S}^{\mathrm{big}}}
\nc{\As}{\mathcal{A}}
\nc{\Ss}{\mathcal{S}}
\nc{\ntens}{\widetilde{\otimes}}
\nc{\hR}{\widehat{R}}
\nc{\nconv}{\mathop{\mbox{\large $\odot$}}}
\nc{\snconv}{\mbox{\scriptsize$\odot$}}
\nc{\ts}{\tilde{s}}
\nc{\sho}{\mathcal{O}}
\nc{\bc}{\begin{cases}}
\nc{\ec}{\end{cases}}
\nc{\slnh}{{\widehat{\mathfrak{sl}}_N}}
\nc{\UA}{U_q'(\slnh)}
\nc{\cQ}{\mathcal{Q}}
\nc{\Irr}{\mathcal{I}rr}
\nc{\Irrsd}{\mathcal{I}rr_{\mathrm{sd}}}
\nc{\tQ}{\widetilde{\cQ}}
\nc{\bs}{\mathbf{s}}
\nc{\bL}{\mathbb{L}}
\nc{\tg}{\tilde{g}}

\nc{\conv}{\mathbin{\mbox{\large $\circ$}}}
\nc{\Rm}{R^{\mathrm{ren}}}

\nc{\bQ}{\ol{Q}}
\renewcommand{\Im}{\on{Im}}

\nc{\de}{\on{\textfrak{d}}}

\nc{\xmono}{\ar@{>->}}
\nc{\xepi}{\ar@{->>}}
\nc{\db}[1]{\raisebox{-.5ex}[2ex][1.8ex]{$#1$}}
\nc{\wb}[1]{\mbox{$\rule[-1.1ex]{0ex}{2ex}#1$}}
\nc{\univ}{\mathrm{univ}}
\nc{\rM}{{}^*\mspace{-2mu}M}
\nc{\lM}{M^*}
\nc{\uqm}{\uq\smod}
\nc{\tR}{\widetilde{R}_{\gamma,\beta}}
\nc{\tx}{\tilde{x}}
\nc{\bi}{\mathbf{i}}
\nc{\ttau}{\widetilde{\tau}}

\nc{\tEnd}{\on{\widetilde{E}nd}}
\nc{\tHom}{\on{\widetilde{H}om}}

\nc{\K}{{J}}
\nc{\Kex}{{\K}_{\mathrm{ex}}}
\nc{\Kfr}{{\K}_{\mathrm{\ms{.5mu}f\mspace{.01mu}r}}}
\nc{\coro}{\cor}
\nc{\tB}{\widetilde{B}}
\nc{\seed}{\mathscr{S}}
\nc{\qseed}{\mathcal{S}}

\nc{\up}{\mathrm{up}}
\nc{\bfa}{\mathbf{a}}

\newcommand{\Span}{{\rm Span}}
\newcommand{\rl}{\mathsf{Q}}   % root lattice
\newcommand{\wl}{\mathsf{P}}   % weight lattice
\newcommand{\comp}{\Delta_+}
\newcommand{\comm}{\Delta_-}

\newcommand{\hconv}{\mathbin{\mbox{$\nabla$}}}

\newcommand{\sconv}{\mathbin{\mbox{$\Delta$}}}

\nc{\tensp}{\otimes_{_+}\mspace{-1mu}}
\nc{\tensm}{\otimes_{_-}\mspace{-1mu}}

\newcommand{\ri}{{\mspace{1mu}\rm r}}

\newcommand{\oi}{\overline{\iota}}
\nc{\bg}{{\oi_\g}}
\nc{\ms}{\mspace}
\nc{\An}{A_q(\n)}
\nc{\Anw}{A_q(\n(w))}
\nc{\tEs}{\widetilde{E}^*}
\def\max{{\mathop{\mathrm{max}}}}
\nc{\pn}{p_\n}
\nc{\dP}{\mathrm{E}^*}
\nc{\Up}{U_q^+(\g)}
\nc{\Ag}{A_q(\g)}
\nc{\QA}{\mathbf{A}}
\nc{\Pd}{\wl^+}
\nc{\Po}{\wl}
\nc{\De}[1]{\Delta(#1)}
\nc{\rt}{\ri}
\nc{\prtl}{\rtl_+}
\nc{\nrtl}{\rtl_-}
\nc{\lt}{\mathrm{l}}
\nc{\wtl}{\wt_\lt}
\nc{\wtr}{\wt_\rt}
\nc{\Cmp}{\comp}
\nc{\Cmm}{\comm}
\nc{\Cm}{\Delta}
\nc{\bolda}{{\bf a}}
\nc{\boldb}{{\bf b}}
\nc{\boldc}{{\bf c}}
\nc{\boldd}{{\bf d}}
\nc{\bolde}{{\bf e}}
\nc{\boldg}{{\bf g}}
\nc{\qt}{\quad\text}
\nc{\Cwt}{\widetilde{\mathcal C}_w}
\nc{\tCw}{\widetilde{\mathcal C}_w}
\nc{\Cw}{{\mathcal C}_w}
\nc{\forg}{\mathrm{for}}
\nc{\ake}[1][2ex]{\rule[-1.5ex]{0ex}{#1}}
\nc{\akew}[1][2ex]{\rule[-1ex]{#1}{0ex}}
\nc{\expR}[1][{\shs}]{\mathbf{g}^R_{#1}}
\nc{\expL}[1][{\shs}]{\mathbf{g}^L_{#1}}
\nc{\pR}{\phi^R}
\nc{\pL}{\phi^L}
\nc{\tAnw}{\widetilde{A_q}(\n(w))}
\nc{\shf}{\mathscr{F}}
\newcounter{myc}

\nc{\mcluster}{\mathscr C}

\newlength{\mylength}
\title
{Laurent phenomenon and simple modules of quiver Hecke algebras}

\author[M. Kashiwara and  M. Kim]{
Masaki Kashiwara$^{1}$ and  Myungho Kim$^{2}$}

\address[Masaki Kashiwara]{Research Institute for Mathematical Sciences, Kyoto University,
Kyoto 606-8502, Japan \& Korea Institute for Advanced Study, Seoul 02455, Korea }
         \email{masaki@kurims.kyoto-u.ac.jp}

\address[Myungho Kim]{Department of Mathematics, Kyung Hee University \\ Seoul 02447, Korea}
         \email{mkim@khu.ac.kr}

\thanks{$^1$ This work was supported by Grant-in-Aid for
Scientific Research (B) 22340005, Japan Society for the Promotion of
Science.}
\thanks{$^2$ This work was supported by the National Research Foundation of
Korea(NF) grant funded by the Korea government(MSIP) (No. NRF-2017R1C1B2007824).}

 %quantum affine algebra}


\date{November 3, 2018}
\begin{document}

\begin{abstract}
In this paper, we study consequences of  the 
results of Kang et al.\ (\cite{KKKO18}) on  a
monoidal categorification of the unipotent quantum coordinate ring $A_q(\mathfrak{n}(w))$ together with the Laurent phenomenon of cluster algebras.
We show that if a simple module $S$ in the category $\mathcal C_w$ strongly commutes with all the cluster variables in a  cluster   $[ \mathscr C]$, 
then $[S]$ is a cluster monomial in $[ \mathscr C ]$.
If $S$ strongly commutes with cluster variables  except exactly one cluster variable $[M_k]$, then $[S]$ is either a cluster monomial in $[\mathscr C ]$ or a cluster monomial in $\mu_k([ \mathscr C ])$.
We give a new proof of the fact that the upper global basis is a common triangular basis (in the sense of Fan Qin (\cite{Qin17})) of the localization  $\widetilde A_q(\mathfrak{n}(w))$ of $A_q(\mathfrak{n}(w))$ at the frozen variables.
A characterization on the commutativity of a simple module $S$ with cluster variables in a cluster $[  \mathscr C]$ is given in terms of the denominator vector
of $[S]$ with respect to the  cluster   $[ \mathscr C]$.

\end{abstract}
\maketitle
\tableofcontents

\section*{Introduction}

The notion of cluster algebras was  introduced by Fomin-Zelevinsky (\cite{FZ02}). 
A cluster algebra is a $\Z$-subalgebra of a rational function field
generated by a special set of generators called the {\it cluster variables}, which are  grouped
into overlapping subsets, called  the {\it clusters} (\cite{FZ02}). A \emph{seed} is a pair of a cluster and  an \emph{exchange matrix} which encodes how one can produce  a new seed from an old one. This procedure is called  \emph{mutation}. 
The celebrated \emph{Laurent phenomenon} (\cite{FZ02}) asserts that every cluster variable can be written as a Laurent
polynomial of cluster variables in any other cluster. 

A monoidal categorification of a cluster algebra $A$ is an abelian monoidal category $\mathcal C$ such that its Grothendieck ring $K(\mathcal C)$ is isomorphic to $A$ as a ring and $\mathcal C$ is required to satisfy several additional properties  reflecting the cluster algebra structure of $A$ (\cite{HL10}).
One of such requirements is that the cluster monomials  (products of cluster variables in a cluster) belong to the set of
 the  classes of simple objects of the category $\mathcal C$.
As observed in \cite{HL10}, this condition and the Laurent phenomenon  immediately implies that the Laurent polynomial expansion  in  a cluster  of the  class of an object of $\mathcal C$ has non-negative coefficients. 
In particular, every cluster variable has a non-negative Laurent expansion in cluster variables of another cluster, which was one of the  central conjectures in the beginning of the cluster algebra theory.  
(Even this non-negativity turned out to be true in general (\cite{LS13}), but the method in \cite{LS13} is totally different from the above.)
It is an example  of  the synergy between monoidal categorification and Laurent phenomenon.

 In \cite{BZ05}, Berenstein-Zelevinsky introduced
the notion of  quantum cluster algebras  which is a $q$-analogue of cluster algebras. 
 The \emph{unipotent quantum coordinate ring $\Anw$}  associated with a Weyl group element $w$ has a structure of quantum cluster algebra (\cite{GLS}).  
It turns out that the quantum cluster monomials of $\Anw$ belong to the upper global basis (\cite{KKKO18}).
This result is accomplished by showing that
the cluster monomials in $\Anw$ is the 
class of simple modules in the category $\mathcal C_w$, where $\mathcal C_w$ is a certain monoidal category 
of finite-dimensional graded modules over a 
\emph{symmetric quiver Hecke algebra} (\cite{KKKO18}).  
Recall that the quiver Hecke algebras (or Khovanov-Lauda-Rouquier algebras) categorify the half of a quantum group (\cite{KL09, R08}), and 
when they are symmetric and the base field is of characteristic zero, the upper global basis corresponds to 
the set of  the isomorphism classes of self-dual simple modules (\cite{VV09, R11}).
Thus the main result in \cite{KKKO18} is 
that the category $\mathcal C_w$ is a monoidal categorification of 
 the quantum cluster algebras $\Anw$.
 In the course of proof, the \emph{$R$-matrices} between modules over symmetric quiver Hecke algebras played a crucial role.

\medskip

 The purpose of this paper  is a study of consequences coming from the 
 monoidal categorification of $\Anw$  and the Laurent phenomenon of cluster algebras.

 We start the application of Laurent  phenomenon  by
proving 
that if a simple module $S$ in $\mathcal C_w$ strongly commutes with all the cluster variables  
 in a given  monoidal cluster  
$\{M_j\}_{j}$  (i.e., $S\conv M_j$ are simple), then the class $[S]$ is a cluster monomial in the cluster  up to a power of $q$
(Theorem~\ref{th:clusterch}).
It follows that  if an upper global basis element $b$ of $\Anw$ $q$-commutes with the cluster variables $\{[M_j] \}$ in a cluster, 
then $b$ is a cluster monomial in the cluster
 up to a power of $q$.
 We can say that this seemingly non trivial result is 
 an  immediate consequence of the monoidal categorification in \cite{KKKO18} and the Laurent phenomenon. 
On a similar line of thought, we show in the last section that  if a simple object $X$ strongly commutes with all the cluster variables in a cluster  $[ \mcluster  ]$ except exactly one cluster variable $M_k$, 
then $[X]$ is either a cluster monomial in the cluster $[ \mcluster  ]$ or a cluster monomial in the mutation $\mu_k([ \mcluster  ])$ of $[ \mcluster  ]$ 
(Theorem~\ref{th:clusterch1}).

We investigate  further properties on the quantum cluster algebra $\tAnw$ in Section~\ref{sec:tri},
where  $\tAnw$ is the localization of $\Anw$ by the set of frozen variables. 
By the  \emph{localization of a monoidal category via a real commuting family of central objects}, which is a   procedure developed  in \cite{Local},
one can  identify the Grothendieck ring $K(\Cwt)$ of the localization $\Cwt$ of $\mathcal C_w$ with the quantum cluster algebra $\tAnw$
(Section~\ref{subsec:Loc}). 
We show that the  homogeneous  degrees of the  $R$-matrices between a cluster  monomial and cluster variables 
gives a natural order $\prec_{\shs}$ on the set $\Z^\K$, 
 the set of exponents of Laurent cluster monomials 
in each  monoidal  seed $\shs$. 
It turns out that for a simple module $S$ of $\mathcal C_w$, the Laurent cluster monomial expansion of $[S]$ in a  quantum seed $[\shs]$   has a unique maximal term and a unique minimal term with respect to $\prec_{\shs}$. Indeed, these maximal and minimal terms correspond to the head and the socle, respectively, when we multiply a cluster monomial to $S$   from the right
such that all its composition factors are again cluster monomials in $\shs$ (Lemma~\ref{lem:Mab}). 
By taking the degrees, we obtain 
natural maps $\expR$ and $\expL$, both 
from the set of the classes of self-dual simple  objects  in $\Cwt$ into $\Z^\K$. 
We show that the maps  $\expR$ and $\expL$ are bijective  (Corollary \ref{cor:parametrization}, Corollary \ref{cor:parametrizationL}).  
 This result is established in several steps: 
first, we show that  $\expR$  are bijective for a  family of special initial seeds. 
Then we calculate the transition map $\phi^R_{\mu_k(\shs),\shs}$ 

between 
 $\expR[\mu_k(\shs)]$ and $\expR$  which turns out to be  bijective,
so that   $\expR$ is bijective for any seed $\shs$.
Note that the map $\phi^R_{\mu_k(\shs),\shs}$ is called the \emph{tropical transformations}  (see \cite{FZ07, FG09, Qin17}) and we recover its formula using the results in \cite{KKKO18}.
Now the bijectivity of $\expL$ follows from the fact that every simple object has a right dual and a left dual  in the category $\Cwt$, which is proved in \cite{Local}, 
and the equality
$\expR(L)+\expL(R)  =0$,   which holds  for any dual pair $(L,R)$ 
 of simples in  $\Cwt$ 
(Lemma~\ref{lem:expanddual}).
 We  emphasize that the maps $\expL$ and $\phi^L_{\mu_k(\shs),\shs}$, which seem new, are as important as $\expR$ and $\phi^R_{\mu_k(\shs),\shs}$.

These properties are strongly related with  the notion of \emph{common triangular  bases} introduced by  Fan Qin \cite{Qin17}. 
 For example, 
 the order $\prec_{\shs}$ is the same as the  \emph{dominance order} given in \cite{Qin17} and 
the bijectivity of $\expR$ is one of the four conditions presented in \cite{Qin17} 
 for the upper global basis  to be a  common triangular basis. 
  Recently, Casbi (\cite{Casbi}) studied
a relationship between the dominance order and the order given by cuspidal decompositions of simples in $R\gmod$, when $R$ is of type $A_n$. 

 In section~\ref{subsec:common}, we give  a new proof of the statement:
the upper global basis of $\tAnw$ is a common triangular basis of the cluster algebra $\tAnw$ (Proposition~\ref{prop:triangular}).
Although this is not a new result (because it can be obtained from \cite{KKKO18}   together with \cite[Theorem 6.1.6]{Qin17} 
as explained  at the end of \cite[Section 1.2]{Qin17}), our approach here  is more direct and enlightening.

 In the last section, 
we show that if the class of a simple module $X$ is divisible by the  class of a real simple module $M$ in the Grothendieck ring, 
then $X$ is indeed isomorphic to a convolution product $Y \conv M$ for some simple module $Y$ (Theorem \ref{thm:commcv}). 
 Together with the Laurent phenomenon, this result helps us  to 
rewrite  the condition that  a simple module $X$ in $\mathcal C_w$ strongly commutes with a cluster variable  in $[ \mcluster  ]$, in terms of 
the \emph{denominator vector} $\boldd(X)$ of $[X]$ with respect to the  cluster   $[ \mcluster  ]$. 
Note that the denominator vector $\boldd(X)$, which is introduced in \cite{FZ02}, can be read easily from the Laurent polynomial expansion of $[X]$ in the  cluster   $[ \mcluster  ]$. 
The characterization enables us  to  provide new proofs of some conjectures in \cite{FZ07} on the denominator vectors  in the case of the cluster algebra $\tAnw|_{q=1}$.

\smallskip

This paper is organized as follows.
In Section~\ref{sec:pre}, we  review quantum unipotent coordinate rings, quantum cluster algebras,  quiver Hecke algebras and  their properties.
In Section~\ref{sec:moncatLaurent}, we recall the main result in \cite{KKKO18} and prove a consequence of it together with the Laurent phenomenon. 
We also prove that the  convolution products of the  cuspidal modules and the cluster variables in the initial  cluster  in a 
 suitable order have simple heads.
In Section~\ref{sec:moncatLaurent}, we collect the results on the   parameterizations  of simple objects in $\tCw$  emphasizing a relevance with the notion of common triangular basis.
In Section~\ref{sec:commden}, we study the relation between the commutativity of a simple module with cluster variables in a  monoidal cluster   and the denominator vectors of the simple module.

\medskip

\noindent
{\bf Acknowledgements.} 
We thank  Se-jin Oh and Euiyong Park   for many fruitful discussions.
The second author thanks  Fan Qin for explaining his results, and 
 gratefully acknowledges the hospitality of
Research Institute for Mathematical Sciences, Kyoto University
during his visits in 2018.

\section{Preliminaries} \label{sec:pre}
 In this paper, we restrict ourselves to the case of
symmetric generalized Cartan matrices and  symmetric quiver Hecke algebras.

\subsection{Quantum groups and unipotent quantum coordinate rings}
Let $I$
be an index set  (for simple roots) and 
let  $(A,\wl,\Pi,\wl^{\vee},\Pi^{\vee})$ be a  \emph{symmetric Cartan datum}.  It consists of a symmetric  \emph{generalized Cartan matrix} $A=(a_{i,j})_{i,j\in I}$,
a free abelian group $\wl$   called the  \emph{weight lattice}, 
a set  $\Pi=\{ \alpha_i \in \wl \mid \ i \in I \}$ of  linearly independent elements called  \emph{simple roots},  the group $\wl^\vee\seteq\Hom_\Z(\wl, \Z)$  called the \emph{coweight lattice},
 and a set $\Pi^\vee=\{ h_i \ | \ i \in I \}\subset \wl^{\vee}$  of \emph{simple coroots}. 
  We assume that 
$\langle h_i,\alpha_j \rangle = a_{ij}$ for all $i,j \in I$,
 and  for each $i \in I$, there exists $ \varpi_i \in \wl$ such that   $\langle h_j,  \varpi_i \rangle =\delta_{ij}$  for all $j \in I$.
The elements $ \varpi_i$ are called the \emph{fundamental weights}.
The free abelian group $\rootl\seteq\soplus_{i \in I} \Z \alpha_i$ is called the
\emph{root lattice}. Set $\rootl^{+}= \sum_{i \in I} \Z_{\ge 0}
\alpha_i\subset\rootl$ and $\rootl^{-}= \sum_{i \in I} \Z_{\le0}
\alpha_i\subset\rootl$. For $\beta=\sum_{i\in I}m_i\al_i\in\rootl$,
we set
$|\beta|=\sum_{i\in I}|m_i|$.
For $\beta\in \rootl^+$ with $n=|\beta|$, we set
$I^{\beta} \seteq \set{\nu = (\nu_1, \ldots, \nu_n) \in I^{n}}%
{ \alpha_{\nu_1} + \cdots + \alpha_{\nu_n} = \beta }.$
Note that the symmetric group $\sym_n=\lan s_1,\ldots, s_{n-1}\ran$ acts on $I^\beta$ by the place permutations.
Set $\mathfrak{h}=\Q \otimes_\Z \wl^{\vee}$.
Then there exists a non-degenerate symmetric bilinear form $(\quad , \quad)$ on
$\mathfrak{h}^*$ satisfying
$ \lan h_i,\lambda\ran=(\alpha_i,\lambda)$ for any $\lambda\in\mathfrak{h}^*$ and $i \in I$.
The  \emph{Weyl group} $W$ of $\g$ is the subgroup  $\lan r_i; \ i\in I\ran \leq GL(\mathfrak h^*)$, where
$
r_i(\la) := \la - \lan h_i, \la \ran \al_i \quad \text{for } \ \la \in \h^*, \ i \in  I.
$
An element  $\mu \in \wl$ is called a \emph{dominant integral weight} if $\lan h_i, \mu \ran \ge 0$ for all $i \in I$. Let us denote the set of dominant integral weights by $\wl^+$.

\medskip
Let $q$ be an indeterminate. 
The quantum group
associated with 
$(A,\wl,\Pi,\wl^{\vee}, \Pi^{\vee})$ is the  associative algebra $\U$ over
$\mathbb Q(q)$ generated by $e_i,f_i$ $(i \in I)$ and
$q^{h}$ $(h \in \wl^\vee)$ with the defining relations

{\allowdisplaybreaks
%\begin{equation*}
\begin{align*}
& q^0=1,\ q^{h} q^{h'}=q^{h+h'} \ \ \text{for} \ h,h' \in \wl^\vee,\\
& q^{h}e_i q^{-h}= q^{\lan h, \alpha_i\ran} e_i, \ \
          \ q^{h}f_i q^{-h} = q^{-\lan h, \alpha_i\ran} f_i \ \ \text{for} \ h \in \wl^\vee, i \in
          I, \\
& e_if_j - f_je_i = \delta_{ij} \dfrac{q^{h_i} -q^{-h_i}}{q- q^{-1}
} \ \  \text{for} \ i,j \in I, 
\\
& \sum^{1-a_{ij}}_{r=0} (-1)^r \left[\begin{matrix}1-a_{ij}
\\ r\\ \end{matrix} \right] e^{1-a_{ij}-r}_i
         e_j e^{r}_i =0 \quad \text{ if } i \ne j, \\*
& \sum^{1-a_{ij}}_{r=0} (-1)^r \left[\begin{matrix}1-a_{ij}
\\ r\\ \end{matrix} \right] f^{1-a_{ij}-r}_if_j
        f^{r}_i=0 \quad \text{ if } i \ne j,
\end{align*}
%\end{equation*}
}
where $[k] = \dfrac{q^k-q^{-k}}{q-q^{-1}}$ for $k \in \Z_{\ge 0}$,
$[k]!=[1]\cdots[k]$ and
and $\left[\begin{matrix}n\\ m\\ \end{matrix} \right]
=\dfrac{[n]!}{[m]![n-m]!}$.

Note that 
$U_q(\g) = \bigoplus_{\beta \in \rootl} U_q(\g)_{\beta}$,
where $$U_q(\g)_{\beta}\seteq\set{ x \in U_q(\g)}{\text{$q^{h}x q^{-h}
=q^{\lan h, \beta \ran}x$ for any $h \in \wl$}}.$$

Let $U_q^{+}(\g)$ be the $\Q(q)$-subalgebra of
$U_q(\g)$ generated by $\set{e_i}{i \in I}$.
 Let $U_q^{+}(\g)_{\Z[q^{\pm1}]}$ be the  
$\Z[q^{\pm1}]$-subalgebra
of $U_q^{+}(\g)$ generated by $e_i^{(n)}\seteq e_i^n/[n]!$ ($i\in I$, $n\in\Z_{>0}$).

Let $\Delta_\n$ be the algebra homomorphism $U_q^+(\g) \to U_q^+(\g)
\tens U_q^+(\g)$ given by
$ \Delta_\n(e_i) = e_i \tens 1 + 1 \tens e_i$, where
the algebra structure on $U_q^+(\g) \tens U_q^+(\g)$ is defined by
$(x_1 \tens x_2) \cdot (y_1 \tens y_2) = q^{-(\wt(x_2),\wt(y_1))}(x_1y_1 \tens x_2y_2).$
Set
$$ A_q(\n) = \soplus_{\beta \in  \rl^-} A_q(\n)_\beta \quad \text{ where } A_q(\n)_\beta \seteq \Hom_{\Q(q)}(U^+_q(\g)_{-\beta}, \Q(q)).$$
Then  $A_q(\n)$ is an algebra with the multiplication given by
$(\psi \cdot \theta)(x)= \theta(x_{(1)})\psi(x_{(2)})$,
when $\Delta_\n(x)=x_{(1)} \tens x_{(2)}$ in  Sweedler's notation.

Let us denote by $\An_{\Z[q^{\pm 1}]}$ the $\Z[q^{\pm 1}]$-submodule of $\An$
 consisting of $ \psi \in \An$ such that 
$ \psi\bl U_q^{+}(\g)_{\Z[q^{\pm1}]}\br\subset\Z[q^{\pm1}]$. 
Then it is a $\Z[q^{\pm 1}]$-subalgebra of $\An$.
 Moreover, the {\em upper global basis}  $\B^\up$ is 
a $\Z[q^{\pm1}]$-module basis of $\An_{\Z[q^{\pm 1}]}$. 
For the detail of the upper global basis, see \cite[Section 1.3]{KKKO18}.
For $w \in W$, let $\Anw_{\Z[q^{\pm 1}]}$ be the $\Z[q^{\pm 1}]$-submodule of $\An_{\Z[q^{\pm 1}]}$
 consisting of  the elements $\psi \in \An_{\Z[q^{\pm 1}]}$ such that 
$e_{i_1}\cdots e_{i_n} \psi =0$  for any $\beta\in \bl\rtl^+\cap w \rootl^+\br
 \setminus \{0\}$ and
any sequence $(i_1,\ldots,i_n) \in I^\beta$. 
Then $\Anw_{\Z[q^{\pm 1}]}$  is a $\Z[q^{\pm 1}]$-subalgebra of $\An_{\Z[q^{\pm 1}]}$ (\cite[Theorem 2.20]{KKOP18}). 
We call this algebra the \emph{quantum unipotent coordinate ring associated with $w$.}
The set $\mathbf B^\up(w) \seteq\mathbf B^\up \cap \Anw$ forms a $\Z[q^{\pm 1}]$-basis of $\Anw_{\Z[q^{\pm 1}]}$ (\cite{Kimu12}).

For a dominant weight $\la \in \wl^+$, 
let $V(\la)$ be the irreducible highest weight $\U$-module with highest weight vector $u_\la$ of weight $\la$. 
Let $( \ , \ )_\la$ be the non-degenerate symmetric bilinear form on $V(\la)$ such that $(u_\la,u_\la)_\la=1$ and $(xu,v)_\la=(u,\varphi(x)v)_\la$ for $u,v \in V(\la)$ and $x\in\Uq$,
where $\varphi$ is  the 
algebra antiautomorphism on $\U$  defined 
by $\varphi(e_i)=f_i$, $\varphi(f_i)=e_i$  and $\varphi(q^h)=q^h$.
For each $\mu, \zeta \in W \la$,
the \emph{unipotent quantum minor} $D(\mu, \zeta)$ is an element in $A_q(\n)$ given by
 $D(\mu,\zeta)(x)=(x u_\mu,u_\zeta)_\la$ 
for $x\in U_q^+(\g)$,
where $u_\mu$ and $u_{\zeta}$ are the extremal weight vectors in $V(\la)$ of weight $\mu$ and $\zeta$, respectively.
For each reduced expression $\underline{w}=s_{i_1}s_{i_2}\cdots s_{i_l}$ of $w \in W$, define 
$w_{\le k}\seteq s_{i_1}\cdots s_{i_k}$  and $w_{<k}\seteq s_{i_1}\cdots s_{i_{k-1}}$ for $1 \le k \le l$.
Then the set $\set{D(w_{\le k}\varpi_{i_k}, w_{< k }\varpi_{i_k})}{1\le k\le l}$ of unipotent quantum minors generates $\Anw_{\Z[q^\pm 1]}$ as a $\Z[q^{\pm1}]$-algebra.

\subsection{Quantum cluster algebras}\label{subsec: cluster}

%%%%%%%%%%%%%%%%%%%%%%%%%%%%%%%%%%%%%%%%%%%%%%%%%%5

Fix a finite index set $\Kex$ and $\Kfr$ and we call the elements in $\Kex$  and $\Kfr$ \emph{exchangeable indices} and \emph{frozen indices}, respectively.
Set $\K\seteq\Kex\sqcup\Kfr$.
For a given  skew-symmetric integer-valued $\K\times \K$-matrix $L=(\lambda_{ij})_{i,j\in \K}$, 
we define $\mathscr P(L)$ as the  $\Z[q^{\pm 1/2}]$-algebra generated by
a family of elements $\{X_i\}_{i\in \K}$
with the defining relations
$X_iX_j=q^{\lambda_{ij}}X_j X_i \quad (i,j\in \K).\label{eq:Xcom}$
For ${\bf a}=(a_i)_{i\in\K}\in \Z^\K$, set
\eqn
&&X^{\bf a}\seteq q^{1/2 \sum_{i > j} a_ia_j\lambda_{ij}}
\overset{\To[{\ }]}{\prod}_{i\in\K}
 X_i^{a_i} \ \in \ \mathscr F(L),
\eneqn
where 
$\mathscr F(L)$ is the skew field of fractions of $\mathscr P(L)$, and 
$\overset{\To[{\ }]}{\prod}_{i\in\K} X_i^{a_i}=X_{i_1}^{a_{i_1}}\cdots X_{i_r}^{a_{i_r}}$
 with $i_1<\cdots <i_r$ in a fixed total order on $\K$.
Note that
$X^{\bf a}$ does not depend on the choice of a total order.
Then
we have
\eqn
&&X^{\bf a}X^{\bf b}=q^{1/2\sum_{i,j\in\K} a_ib_j\la_{ij}}X^{\bf a+\bf b}
\eneqn 
and
 $\{X^{\mathbf a}\}_{\mathbf{a}\in\Z_{\ge0}^\K}$ is a basis
of $\mathscr P(L)$ as a $\Z[q^{\pm 1/2}]$-module.

Let $\widetilde B = (b_{ij})_{(i,j)\in\K\times\Kex}$ be an  integer-valued
$\K \times \Kex$-matrix whose {\em principal part $B\seteq(b_{ij})_{i,j\in\Kex}$} is  {\em skew-symmetric}. 
The matrix $\wdt B$ is called {\em an exchange matrix.} 
A pair $(L, \widetilde B)$ is called {\em compatible} if there exists a positive integer $d$ such that
$\sum_{k\in\K} \lambda_{ik} b_{kj} =\delta_{ij} d \quad (i\in \K,\;j\in\Kex)$.
In this paper, we treat only the case $d=2$.
Let $(L,  \widetilde B)$ be a compatible pair and $A$ a $\Z[q^{\pm1/2}]$-algebra.

We say that $ \qseed = (\{x_i\}_{i\in\K},L, \widetilde B)$ is
a {\em quantum seed} in $A$ if $\{x_i\}_{i\in\K}$ satisfies that 
$x_ix_j=q^{\la_{ij}}x_jx_i$ for any $i,j\in\K$ and  the the algebra map
$\mathscr P(L)\to A$ given by $X_i\mapsto x_i$ is injective.

In that case, we call the set $\{x_i\}_{i\in\K}$  the {\em cluster} of $\qseed$ and
its elements  the {\em cluster variables}.
The elements $x^{\bf a}$ (${\bf a}\in \Z_{\ge0}^\K$) are called  the {\em quantum cluster monomials}.

Let $\qseed=(\{x_i\}_{i\in\K},L, \widetilde B)$ be a quantum seed in a $\Z[q^{\pm1/2}]$-algebra $A$ which is contained in a skew field $K$ .
 For each $k\in\Kex$, we define
\eqn \label{eq:mutation B}
\mu_k(\widetilde B)_{ij} &:=&
\begin{cases}
  -b_{ij} & \text{if}  \ i=k \ \text{or} \ j=k, \\
  b_{ij} + (-1)^{\delta(b_{ik} < 0)} \max(b_{ik} b_{kj}, 0) & \text{otherwise,}
\end{cases}\\
\mu_k(L)_{ij} &:=&
\begin{cases}
  0 & \text{if}  \ i=j \\
  -\la_{kj}+\displaystyle\sum _{t\in\K} \max(0, -b_{tk}) \la_{tj} & \text{if} \ i=k, \ j\neq k, \\
  -\la_{ik}+\displaystyle\sum _{t\in\K} \max(0, -b_{tk}) \la_{it} & \text{if} \ i \neq k, \ j= k, \\
  \la_{ij} & \text{otherwise,}
\end{cases}
\eneqn
and 
\eqn \label{eq:quantum mutation}
\mu_k(x)_i\seteq
\begin{cases}
  x^{{\bf a}'}  +   x^{{\bf a}''}, & \text{if} \ i=k, \\
 x_i & \text{if} \ i\neq k,
\end{cases}
\eneqn
where  ${\bf a}'=(a_i')_{i\in\K}$ and ${\bf a}''=(a_i'')_{i\in\K}$ are
given by
\eqn
&&a_i'=
\begin{cases}
  -1 & \text{if} \ i=k, \\
 \max(0,b_{ik}) & \text{if} \ i\neq k,
\end{cases} \qquad
a_i''=
\begin{cases}
  -1 & \text{if} \ i=k, \\
 \max(0,-b_{ik}) & \text{if} \ i\neq k.
\end{cases}
\label{eq:aa}
\eneqn
Then the triple 
\eqn
\mu_k(\qseed)\seteq\bl\{\mu_k(x)_i\}_{i\in\K},\mu_k(L),\mu_k(\widetilde B)\br
\eneqn
becomes a new quantum seed in $K$ and we call it
the {\em mutation
of $\qseed$ at $k$}.

   The {\em quantum cluster algebra $\mathscr A_{q^{1/2}}(\qseed)$ associated to the quantum seed $\qseed$} is  the $\Z[q^{\pm 1/2}]$-subalgebra of the skew field $K$ generated by all the  cluster variables in the quantum seeds obtained from $\qseed$ by any sequence of mutations.

A {\em quantum cluster algebra structure} on a $\Z[q^{\pm1/2}]$-algebra $A$ 
is a family $\shf$ of quantum seeds in $A$ such that
\bna
\item for any quantum seed $\qseed$ in $\shf$, $\mathscr A_{q^{1/2}}(\qseed)$
is isomorphic to $A$,
\item for any pair $\qseed,\qseed'$ of quantum seeds in $\shf$,
$\qseed'$ is obtained from $\qseed$ by a  sequence of  mutations,
\item any mutation of a quantum seed in $\shf$ is in $\shf$.
\ee

We call a quantum seed in $\shf$ a quantum seed of the quantum cluster algebra $A$.

\bigskip

The \emph{quantum Laurent phenomenon}  (\cite[Corollary 5.2]{BZ05}) asserts that for any quantum seed $\qseed$, 
every element $x \in \mathscr A_{q^{1/2}}(\qseed)$ can be expressed as
\eqn \label{eq:qLP}
x=\sum_{\bfa\, \in\, \Z^{\Kex}  \, \times \, \Z_{\ge 0}^{\Kfr} }c_\bfa x^\bfa 
\eneqn
in $K$,
where
$X=(x_i)_{i\in \K}$ is the  cluster in $\qseed$, 
and $c_\bfa\in \Z[q^{\pm1/2}]$.

Remark that in \cite{BZ05}, the quantum cluster algebra associated with $\qseed$ is defined as 
the  subalgebra of $K$ generated by all the  cluster variables 
together with the inverses of the frozen variables.
Then it  is the (right) ring of quotient of $\mathscr A_{q^{1/2}}(\qseed)$ with respect to the multiplicative subset consisting of monomials in the frozen variables.

%%%%%%%%%%%%%%%%%%%%%%%%%%%%%%%%%%%%%%%%%%%%%%%%5

\smallskip

Gei\ss,  Leclerc and Schr\"oer showed that the algebra $\Z[q^{\pm1/2}] \otimes_{\Z[q^{\pm1}]} \Anw_{\Z[q^{\pm1}]}$ has a quantum cluster algebra structure (\cite{GLS}).
 In particular, for each choice of reduced expression $\underline{w}=s_{i_1}s_{i_2}\cdots s_{i_l}$, they gave a distinguished quantum seed whose cluster is consisting of 
$D(w_{\le k} \varpi_{i_k}, \varpi_{i_k} )$ for $1\le k\le l$ up to a power of $q^{1/2}$.
For a detailed description, see  subsection \ref{subsec:initial seeds} below.

\subsection{Quiver Hecke algebras}
Let $\cor$ be a base field.
Take  a family of polynomials $(Q_{ij})_{i,j\in I}$ in $\cor[u,v]$
which are of the form 
\eq
Q_{ij}(u,v) =&& \begin{cases}\hs{5ex} 0 \ \ & \text{if $i=j$,} \\[1.5ex]
\pm(u-v)^{-(\al_i,\al_j)}
& \text{if $i \neq j$}
\end{cases}
\label{eq:Q}
\eneq
such that $Q_{i,j}(u,v)=Q_{j,i}(v,u)$.

The (symmetric) {\em  quiver Hecke algebra} (also called,  \emph{Khovanov-Lauda-Rouquier algebra})  $R(\beta)$  at $\beta\in\rtl^+$ associated
with  $(A,\wl, \Pi,\wl^{\vee},\Pi^{\vee})$ and 
$(Q_{ij})_{i,j \in I}$, is the  algebra over $\cor$
generated by the elements $\{ e(\nu) \}_{\nu \in  I^{\beta}}$, $
\{x_k \}_{1 \le k \le n}$, $\{ \tau_m \}_{1 \le m \le n-1}$
satisfying the following defining relations:
{\allowdisplaybreaks
%\begin{equation*} \label{eq:KLR}
\begin{align*}
& e(\nu) e(\nu') = \delta_{\nu, \nu'} e(\nu), \ \
\sum_{\nu \in  I^{\beta} } e(\nu) = 1, \\
& x_{k} x_{m} = x_{m} x_{k}, \ \ x_{k} e(\nu) = e(\nu) x_{k}, \\
& \tau_{m} e(\nu) = e(s_{m}(\nu)) \tau_{m}, \ \ \tau_{k} \tau_{m} =
\tau_{m} \tau_{k} \ \ \text{if} \ |k-m|>1, \\
& \tau_{k}^2 e(\nu) = Q_{\nu_{k}, \nu_{k+1}} (x_{k}, x_{k+1})
e(\nu), \\
& (\tau_{k} x_{m} - x_{s_k(m)} \tau_{k}) e(\nu) = \begin{cases}
-e(\nu) \ \ & \text{if} \ m=k, \nu_{k} = \nu_{k+1}, \\
e(\nu) \ \ & \text{if} \ m=k+1, \nu_{k}=\nu_{k+1}, \\
0 \ \ & \text{otherwise},
\end{cases} \\
& (\tau_{k+1} \tau_{k} \tau_{k+1}-\tau_{k} \tau_{k+1} \tau_{k}) e(\nu)\\*
& =\begin{cases} \dfrac{Q_{\nu_{k}, \nu_{k+1}}(x_{k},
x_{k+1}) - Q_{\nu_{k}, \nu_{k+1}}(x_{k+2}, x_{k+1})} {x_{k} -
x_{k+2}}e(\nu) \ \ & \text{if} \
\nu_{k} = \nu_{k+2}, \\
0 \ \ & \text{otherwise}.
\end{cases}
\end{align*}
%\end{equation*}
}% the end of \allowdisplaybreaks
The algebra $R(\beta)$ is $\Z$-graded with $\deg e(\nu)=0$, $\deg x_ke(\nu)=2$, and $\deg \tau_m e(\nu)=-(\alpha_{\nu_m},\alpha_{\nu_{m+1}})$.

 Note that the choice of $\pm$ in \eqref{eq:Q}
is irrelevant.
Namely, the isomorphism class of the quiver Hecke algebra $R(\beta)$ 
does not depend on the choice of $\pm$.

Let us denote by $R(\beta)\gmod$ the category of finite-dimensional graded $R(\beta)$-modules with homomorphisms of degree $0$. 
For an  $R(\beta)$-module $M$, we set $\wt(M)\seteq -\beta \in \rootl^-$.
For the sake of simplicity,
we  say that $M$ is an $R$-module instead of saying
that $M$ is a graded $R(\beta)$-module for $\beta\in\rtl^+$.
Let us denote by $q$ the grading shift functor: for $M\in R\gmod$, we have
$(qM)_n=M_{n-1}$ by definition.

We also sometimes ignore grading shifts if there is no  danger of confusion.
Hence, for $R$-modules $M$ and $N$,  we sometimes say that
$f\col M\to N$ is a homomorphism if
$f\col q^aM\to N$ is a morphism in $R\gmod$ for some $a\in\Z$.
If we want to emphasize that $f\col q^aM\to N$ is a morphism in $R\gmod$,
we say so.
 We set  
$$\HOM_{R(\beta)}(M,N)\seteq \soplus_{a\in \Z} \Hom_{R(\beta)\gmod}(q^aM,N).$$

For $M \in R(\beta)\gmod$, the space $M^*\seteq \Hom_{\cor}(M,\cor)$ admits an $R(\beta)$-module structure 
 induced  by the $\cor$-algebra antiautomorphism $\psi$ on $R(\beta)$ 
which fixes the generators $e(\nu)$, $x_k$, and $\tau_m$. 
Hence we have a contravariant auto-equivalence $^*$ on $R\gmod$.
We call a simple $R(\beta)$-module  $M$   is \emph{self-dual} if $M\simeq M^*$ in $R\gmod$.

For $\beta, \gamma \in \rootl^+$, 
 set
$e(\beta,\gamma)\seteq \displaystyle\sum _{\nu \in I^{\beta}, \nu' \in I^{\gamma}}
 e(\nu * \nu')  \in  R(\beta+\gamma)$, where $\nu*\nu'$ denotes the concatenation of $\nu$ and $\nu'$.
Then there is a $\cor$-algebra homomorphism 
$R( \beta)\tens R( \gamma  )\to e(\beta,\gamma)R( \beta+\gamma)e(\beta,\gamma)$. 
The \emph{convolution product}
$- \conv - : R(\beta)\gmod \times R(\gamma)\gmod \to R(\beta+\gamma)\gmod$
 is a bifunctor given by  
$$M\conv N\seteq R(\beta + \gamma) e(\beta,\gamma)
\tens_{R(\beta )\otimes R( \gamma)}(M\otimes N).$$
Set $R\gmod \seteq \soplus_{\beta \in \rootl_+} R(\beta)\gmod$.
 Then the category $R\gmod$ is a monoidal category with the tensor product $\conv$ and the unit object $\one\seteq \cor \in R(0) \gmod$.

For $M\in R(\beta)\gmod$ and $N \in R(\gamma)\gmod$, we have 
$(M\conv N)^*\simeq q^{(\beta,\gamma)}(N^*\conv M^*)$.

\subsection{R-matrices and real simple modules}

  For $1\le a<|\beta|$,  we define the {\em intertwiner} $\vphi_a\in R( \beta)$
by
\eqn&&\ba{l}
  \vphi_a e(\nu)=
\begin{cases}
  \bl\tau_ax_a-x_{a}\tau_a\br e(\nu)
  & \text{if $\nu_a=\nu_{a+1}$,} \\[2ex]
\tau_ae(\nu)& \text{otherwise.}
\end{cases}
 \ea \eneqn
For $m,n\in\Z_{\ge0}$,
let us denote by $w[{m,n}]$ the element of $\sym_{m+n}$  defined by
\eqn
&&w[{m,n}](k)=\begin{cases}k+n&\text{if $1\le k\le m$,}\\
k-m&\text{if $m<k\le m+n$,}\end{cases}
\eneqn
and set $\vphi_{[m,n]}\seteq\vphi_{i_1}\cdots\vphi_{i_\ell}$, where $s_{i_1}\cdots s_{i_t}$ is a reduced expression of $w[m,n]$.
 It does not depend on the choice of a reduced expression.

Let
$M$ be an $R(\beta)$-module and $N$ an $R(\gamma)$-module. Then the
map $M\tens N\to N\conv M$ given by
$u\tens v\longmapsto \vphi_{[|\gamma|,|\beta|]}(v\tens u)$
induces 
 an $R( \beta +\gamma)$-module  homomorphism
\eqn &&R_{M,N}\col
M\conv N\To N\conv M. \eneqn

Let $z$ be an indeterminate  which is homogeneous of degree $2$, and
let $\psi_z$ be the graded algebra homomorphism
\eqn
&&\psi_z\col R( \beta )\to \cor[z]\tens R( \beta )
\eneqn
given by
$$\psi_z(x_k)=x_k+z,\quad\psi_z(\tau_k)=\tau_k, \quad\psi_z(e(\nu))=e(\nu).$$

For an $R( \beta )$-module $M$, we denote by $M_z$
the $\bl\cor[z]\tens R( \beta )\br$-module
$\cor[z]\tens M$ with the action of $R( \beta )$ twisted by $\psi_z$.

For a non-zero $R(\beta)$-module $M$ 
and a non-zero $R(\gamma)$-module $N$,
let $s$ be the largest non-negative integer such that the image of
$R_{M_z,N}$ is contained in $z^s(N\conv M_z)$.
We denote by $\Rm_{M_z,N}$
the $R(\beta)$-module homomorphism $z^{-s}R_{M_z,N}$ and call it the \emph{renormalized R-matrix}. 
We denote  by $\rmat{M,N}:= \Rm_{M_z,N}\vert_{z=0} : M\conv N \to N \conv M$ 
and call $\rmat{M,N}$ the \emph{R-matrix}. 
 By the definition, $\rmat{M,N}$ never vanishes. 
We denote by $\Lambda(M,N)$ the homogeneous degree of $\rmat{M,N}$ and set
\eqn
&&\tL(M,N) \seteq \frac{1}{2}\left(\Lambda(M,N) +(\wt(M), \wt(N)) \right) ,\quad 
\de(M,N) \seteq \frac{1}{2}\bl\Lambda(M,N)+\Lambda(N,M) \br.
\eneqn
 Then we have 
\eq
&&\tL(M,N), \quad\de(M,N)\in \Z_{\ge 0}\label{eq:Lapos}
\eneq
  by \cite[Proposition\;1.10\;(iii)]{KKK18} and \cite[Lemma\;3.2.1]{KKKO18} for any simple modules
 $M$ and $N$. 
 For more properties on $\tL(M,N)$ and $\de(M,N)$, see  \cite[Section\;3.2]{KKKO18}.

We say that two simple $R$-modules $M$ and $N$ {\em strongly commute} if $M\conv N$ is simple. 
In such a case, $\rmat{M,N}\cl q^{\La(M,N)}M\conv N\to N\conv M$ is an isomorphism. 
If a simple module $M$  strongly commutes with itself, then  $M$ is called {\em real}.
Assume that $M$ and $N$ are simple modules and  one of them is real. 
Then $M$ and $N$ strongly commute if and only if $M \conv N$ is isomorphic to $N \conv M$ up to a grading shift, which is equivalent to the condition  $\de(M,N)=0$ (see, \cite[Lemma 3.2.3]{KKKO18}).

For simple $R$-modules $M$ and $N$, we denote by  $M \hconv N$   the head of
$M \conv N$ and  by $M \sconv N$ the socle of $M \conv N$ .

Assume that  $M$ and $N$ are self-dual simple modules and one of them is real.
Then
$q^{\tL(M,N)} M \hconv N$
is a self-dual simple module.

\smallskip
Let $\{M_i\}_{i\in \K}$ be a  mutually strongly commuting family of 
 self-dual  real simple modules.
For an element $\bolda=(\bolda_i)_{i\in J} \in\Z^{J}_{\ge 0}$,
there exists a unique self-dual simple module which is a  convolution   product of modules $M_i$ with multiplicities $\bolda_i$. We will denote it by $M(\bfa)$ in short. 
For any  sequence $(k_1,\ldots, k_m)$ 
in which each $i \in \K$ appears $\bolda_i$ times, 
we have 
\eq \label{eq:selfdualconvs} 
M(\bfa) \simeq q^{\sum_{1 \le p <q\le m} \tLa(M_{k_p},M_{k_q})}M_{k_1} \conv \cdots \conv M_{k_m}\eneq 
in $R\gmod$, 
since the latter is a self-dual simple module.

Note that we have
$$M(\bolda)\conv M(\boldb)\simeq 
q^{\;-\sum\limits_{i,j\in\K}\bolda_i\boldb_j\tLa(M_i,M_j)}\;M(\bolda+\boldb)
\qt{for any $\bolda$, $\boldb\in \Z^{\K}_{\ge 0}$.}$$

\section{Monoidal categorification and Laurent phenomenon} 
\label{sec:moncatLaurent}
\subsection{Monoidal categorification of $\Anw$}
 The Grothendieck  group 
$K(R\gmod)$ has a $\Z[q^{\pm1}]$-algebra structure.
The multiplication structure on $K(R\gmod)$ is induced by the convolution product and the action of $q$ is induced by the grading shift functor $q$.
Moreover the isomorphism classes of self-dual simple modules form
a  $\Z[q^{\pm1}]$-module basis of $K(R\gmod)$.

There is a $\Z[q^{\pm1}]$-algebra isomorphism between the Grothendieck ring $K(R\gmod)$ and $\An_{\Z[q^{\pm1}]}$ (\cite{KL09}). 

\medskip

If we assume further that $\cor$ is of characteristic $0$, then  the  basis consisting of the classes of self-dual simple modules in $R\gmod$ corresponds to the upper global basis $\mathbf B^\up$ under the isomorphism (\cite{VV09,R11}).  
We have $[M^*]=\overline{[M]}$ for $M \in R\gmod$, where $-$ denotes the \emph{bar involution} which is the $\Z$-linear involution on $K(R\gmod)$ sending $q\mapsto q^{-1}$ and fixing the  
classes of  self-dual simple modules.

In particular,  the quantum unipotent minors $D(\la, \mu)$  belong to $\mathbf B^\up$ and we denote the corresponding self-dual simple modules $M(\la,\mu)$ and call it a 
\emph{determinantial module}.
 Even if ${\rm char}(\cor) \neq 0$, 
there exists a unique simple module $M(\la,\mu)$
  such that $[M(\la,\mu)]=D(\la,\mu)$ (\cite[Proposition 4.1]{KKOP18}).

For  an $R(\beta)$-module $M$,  define
\eqn
\mathrm W (M)\seteq \set{\gamma \in  \rootl^+\cap (\beta-\rootl^+)}{e(\gamma, \beta-\gamma) M \neq 0},&&\\
\mathrm W^* (M)\seteq \set{\gamma \in  \rootl^+\cap (\beta-\rootl^+)}{e(\beta-\gamma,\gamma) M \neq 0}.&&
\eneqn
For $w \in W$, we denote by $\mathcal C_w$ the full subcategory of $R\gmod$ whose objects $M$ satisfy $\mathrm W(M) \subset \rootl^+\cap w \rootl^-$. 
Then the category $\mathcal C_w$ is the  smallest monoidal abelian full subcategory of $R\gmod$ which is stable under taking subquotients, extensions, grading shifts, and  contains  the self-dual simple modules  $\set{M(w_{\le k}\varpi_{i_k},w_{<k}\varpi_{i_k})}{1\le k\le l}$, where $\underline{w}=s_{i_1}s_{i_2}\cdots s_{i_l}$ is a reduced expression  of $w$ (\cite[Theorem 2.20]{KKOP18}).  
The Grothendieck ring $K(\mathcal C_w)$ is isomorphic to $\Anw_{\Z[q^{\pm1}]}$.
Moreover, the set of
  the classes of self-dual simple modules coincides with
the upper global basis,
whenever $\cor$ is of characteristic $0$. 

We set
\eqn
K(R\gmod)\vert_{q=1}&&\seteq
K(R\gmod)/(q-1)K(R\gmod),\\
K(\Cw)\vert_{q=1}&&\seteq
K(\Cw)/(q-1)K(\Cw).
\eneqn
They are commutative rings.

In \cite{KKKO18}, it is shown that the category $\mathcal C_w$ gives a \emph{monoidal categorification of the cluster algebra $\Anw$}.
The following is one of  the main results
in \cite{KKKO18}.
\begin{theorem} \label{thm:moncat}
 Let $\qseed=(\{X_i\}_{i\in J}, L,\widetilde B)$ be a quantum seed of the quantum cluster algebra  
$\Z[q^{\pm1/2}] \tens_{\Z[q^{\pm1}]}\Anw$. 
 Here
$\K=\Kex \sqcup \Kfr$ is the index set of cluster variables, 
$\Kex$ is the index set of
exchangeable cluster variables, and $\Kfr$ is the index set of frozen cluster
variables, 
 and $\widetilde B = (b_{i,j})_{i\in \K, j \in \Kex }$ is 
an exchange matrix \ro see subsection~{\rm\ref{subsec: cluster}}\rf.

\bni
\item 
 The quantum seed $\qseed$  has the form 
\eq
[\seed]\seteq(\{q^{-\frac{1}{4}(\wt(M_i), \wt(M_i))}[M_i]\}_{i\in\K}, -\La,\widetilde B),\label{def:associated_seed}
\eneq
where 
 $ \mcluster  \seteq\{M_i\}_{i\in \K}$  is a 
 mutually strongly commuting family of real simple modules 
in $\mathcal C_w$  and $\La=\bl\La(M_i,M_j)\br_{i,j\in \K}$.

\item For each $k \in \Kex$, there exists a self-dual simple module $M_k'$ in $\mathcal C_w$ such that there exists an exact sequence in $\mathcal C_w$ 
\eqn 
&&\hs{5ex}0 \to \mathop{\conv}\limits_{b_{i,k}>0} M_i^{\circ b_{i,k}} \to 
M_k \conv M_k' \to \mathop{ \conv}\limits_{b_{i,k}<0} M_i^{\circ( -b_{i,k})} \to 0\qt{up to grading shifts.}
\eneqn
We call the module $M_k'$  the \emph{mutation of $M_k$ at $k$}.
\end{enumerate}
\end{theorem}

\smallskip
\begin{remark}
Recall that the above theorem is proved in \cite{KKKO18} under the assumption  ${\rm char} ( \cor)=0$. 
But this  assumption can be dropped in  the above theorem, 
since \cite[Proposition 4.6]{KKOP18} replaces \cite[Theorem 4.2.1]{KKKO18}, 
which is the only place using the assumption in the course of the proof. 
\end{remark}

If $[\seed]$ in \eqref{def:associated_seed} is a quantum seed
and satisfies (ii) in the theorem above,  then we call the triple
$\seed =(\{M_i\}_{i\in\K}, -\La,\widetilde B)$ a \emph{monoidal seed in $\Cw$},
and $\mcluster=\{M_i\}_{i\in \K}$  
a \emph{ monoidal cluster   in  $\mathcal C_w$}.
By Lemma \ref{lem:invertible} below, the  monoidal seed $\seed$
is completely  determined by the monoidal cluster $\mcluster$.

 We call $\mu_k( \mcluster)\seteq\{M_i'\}_{i\in \K}$ 
the  \emph{mutation}  of $ \mcluster $ at $k$,
where $M_i'=M_i$ $(i\neq k)$ and $M_k'$ is as in  Theorem~\ref{thm:moncat}~(ii) above.

\subsection{Laurent phenomenon}

Let $ \mcluster  = \{M_i\}_{i\in \K}$ be a   monoidal cluster   in $\mathcal C_w$. 
For $\bolda\in\Z_{\ge0}^\K$, we denote by
$M_{ \mcluster }(\bolda)$ the module 
$M(\bolda)$ defined in \eqref{eq:selfdualconvs} starting from $ \mcluster =\{M_i\}_{i\in \K}$.
If there is no afraid of confusion, we  simply write 
 $M(\bolda)$ instead of $M_{ \mcluster }(\bolda)$. 

The subalgebra  $\Z[q^{\pm1}][[M_i] ; i\in \K]$ of $\Anw$
is  isomorphic to the polynomial ring with the $[M_i]$'s as variables. 
The  quantum Laurent phenomenon  claims that
$K(\mathcal C_w)$ is contained in the quantum Laurent polynomial ring
$\Z[q^{\pm1}][[M_i] ; i\in \Kfr][[M_i]^{\pm1}  ; i\in \Kex]$ in the skew field of fractions of $\Anw$. 

 Hence, for a module $X$ in $\mathcal C_w$, 
there exists  
$\bolda \in \Z_{\ge 0}^\K$ with $a_j =0 $ for $j \in \Kfr$ such that 
\eq \label{eq: expansion}
[X \conv M(\bolda)] =[X]\cdot [M(\bolda)] = \sum_{1 \le s \le r} c(\boldb(s)) [M(\boldb(s))]
\eneq
for some $r \ge 0$, $\boldb(s) \in \Z_{\ge 0}^\K$ and $c(\boldb(s)) \in \Z[q^{\pm 1/2}]$. 
Since  $[X \conv M(\bolda)]$ is the class of a module in $\mathcal C_w$ and   each $[M(\boldb(s))]$ is the  class of  a real
simple module, the coefficient $c(\boldb(s))$ lies in  $\Z_{\ge 0}[q^{\pm 1}]$ for all $1\le s \le r$.

Summing up, the following lemma is an immediate consequence of the quantum Laurent phenomenon.

\Lemma\label{lem:LP}
For a module $X$ in $\mathcal C_w$, 
there exist  
$\bolda \in \Z_{\ge 0}^\K$ with $\bolda_j =0 $ for $j \in \Kfr$,
$r\in\Z_{\ge0}$, $\boldb(s) \in \Z_{\ge 0}^\K$ and $c_s\in\Z$ $(1\le s\le r)$
such that
\eq \label{eq: expansion1}
[X \conv M(\bolda)] =[X]\cdot [M(\bolda)] =
\sum_{1 \le s \le r} q^{c_s}[M(\boldb(s))].
\eneq
\enlemma

In the sequel, we write
\eq
&&K(\Cw)[ \mcluster ^{-1}]\seteq\Z[q^{\pm1}][[M_i]^{\pm1} ; i\in \K],
\label{eq:Lrpr}
\eneq
 for a monoidal cluster $\mcluster = \{M_i\}_{i\in \K}$. 

The following theorem is an immediate  consequence of the lemma above.

\begin{theorem}[cf.\;Theorem~\ref{th:clusterch1}] \label{th:clusterch}
Let $X$ be a simple module in $\mathcal C_w$ and let $ \mcluster  =\set{M_i}{i \in \K}$ be a  monoidal cluster   in $\mathcal C_w$. If $X$ strongly commutes with $M_i$ for all $i \in \K$, then $[X]$ is a cluster monomial in the  cluster 
$[ \mcluster ]$  up to a power of $q$. In particular, $X$ is a real simple module.
\end{theorem}
\begin{proof} 
We have \eqref{eq: expansion1} for some $\bolda\in \Z_{\ge 0}^\K$.
Since $X$ strongly commutes with $M_i$ for all $i \in \K$, the convolution 
$X \conv M(\bolda)$ is simple (\cite[Lemma 3.2.3]{KKKO18}).
Since the self-dual simple modules form a $\Z[q^{\pm1}]$-basis of $K(\mathcal C_w)$, we conclude that   $X \conv M(\bolda) \simeq  q^c M(\boldb)$  for some $c \in \Z$ and $\boldb \in \Z_{\ge 0}^\K$.
Recall that for  every  $k \in \Kex$ and $i \in \K$, we have 
$\de(M_k', M_i) =\delta_{ki}$, where $M_k'$ is the mutation of $M_k$ (\cite[Proposition 7.1.2\;(d)]{KKKO18}). 
It follows that $b_k=\de(M_k', M(\boldb)) = \de(M_k', X)+a_k $
so that $b_k-a_k \ge 0$ for all $k \in \Kex$.
Thus we conclude that 
$$X \simeq M(\boldb -\bold a)$$
 up to a grading shift 
 by \cite[Corollary 3.7]{KKKO15},
 as desired.
\end{proof}

\subsection{Initial seeds and dual PBW basis}
\label{subsec:initial seeds}

\medskip
Let $\underline{w} = s_{i_1} \cdots s_{i_l}$ be a reduced expression of $w\in W$,
and  set ${\bold i}=(i_1,\ldots,i_l)$. 
 Recalling that $w_{\le k}\seteq s_{i_1}\cdots s_{i_k}$, $w_{<k}\seteq s_{i_1}\cdots s_{i_{k-1}}$, we set $\beta_k:=w_{<k }(\al_{i_k})$ for $1 \le \ k \le l$.
Then  $\Delta_+ \cap w \Delta_- = \set{\beta_k}{1\le k\le l}$, 
 where $\Delta_+$ 
and $\Delta_-$ denote the set of positive roots 
and the set of negative roots, respectively.
 Let $\preceq$ be a \emph{convex order} on the set $\Delta_+$ satisfying
\eqn \beta_1\prec \beta_2 \prec \cdots \prec \beta_l \prec \gamma \quad \text{for any} \ \gamma \in  \Delta_+ \cap w\Delta_+.
\eneqn 
 See \cite{TW16,KKOP18} for details of convex orders on $\Delta_+$. 

Define
\eqn
&&S_k:=M(w_{\le k} \varpi_{i_k}, w_{< k}\varpi_{i_k}) \qtext M_k:=M(w_{\le k} \varpi_{i_k}, \varpi_{i_k}) \quad \text{for} \quad 1\le k \le l.
\eneqn
Then $S_k$ is a self-dual $R(\beta_k)$-module such that 
$\mathrm W (S_k ) \subset {\rm span}_{{\mathbb R}_{\ge 0}}\set{\gamma \in \Delta_+}{\gamma \preceq \beta_k}$   (i.e., $S_k$ is a \emph{$\preceq$-cuspidal} module). 
 For any sequence $\bolda = (\bolda_i)_{1 \le i \le l}  \in \Z_{\ge 0}^l$, the convolution product
$$P_{\bold i}(\bolda):=q^{\frac{1}{2}\sum_{k=1}^{l} \bolda_k(\bolda_k-1)} S_l^{\circ \bolda_l} \conv \cdots \conv S_1^{\circ \bolda_1}$$
 has a self-dual simple head.
Moreover, every self-dual simple module in $\mathcal C_w$ 
is isomorphic to $\hd (P_{\bold i}(\bolda))$ for some $\bolda  \in \Z_{\ge 0}^l$.
 Moreover, such an $\bolda$ is unique. 

Hence, $\{[P_{\bold i}(\bolda)]\}_{\bolda  \in \Z_{\ge 0}^l}$
is a $\Z[q^{\pm1}]$-basis of $K(\Cw)$.  It is called the dual PBW basis. 

\medskip
 
It is shown in  \cite{GLS} that  the triple  
 $(\{q^{-(d_k,d_k)/4}[M_k]\},-\La,\tB)$ is an initial quantum seed for the quantum cluster algebra $\Z[q^{\pm1/2}] \otimes_{\Z[q^{\pm1}} \Anw_{\Z[q^{\pm1}]}$, where
$d_k:=\wt(M_k)=w_{\le k}\varpi_{i_k}-\varpi_{i_k}$, $\La=(\La(M_i,M_j))_{1\le i,j\le l}$, and $\tB$ is $l \times (l-n)$-matrix given in \cite[Definition 11.1.1]{KKKO18}. 
In this case,  the index set of cluster variables are  $\K=\{1,2,\ldots, l\}$, $\Kfr=\set{k\in\K}{k_+=l+1}$, where 
$k_+\seteq\min\bl \set{s}{k<s\le l,\; i_k=i_s}\sqcup\{l+1\}\br$,  and $\Kex=\K \setminus \Kfr$,
 $n=\sharp\{i_1,\ldots,i_\ell\}=\sharp\Kfr$.

\bigskip
 We shall introduce the following notion in order to describe a property of
$S_k$'s and $M_k$'s.

\begin{definition}
A sequence $(L_1,\ldots,L_r)$ of real simple modules  in $R \gmod$ is called a \emph{normal sequence} if the composition of the 
r-matrices 

\eqn
\rmat{L_1,\ldots,L_r}\seteq
\displaystyle\prod_{1\le i <k \le r} \rmat{L_i,L_k} =&&(\rmat{L_{r-1},L_r})  \circ \cdots \circ (\rmat{L_2,L_r}\circ \cdots \circ \rmat{L_2,L_3})  \circ (\rmat{L_1,L_r} \circ \cdots  \circ \rmat{L_1,L_2}) 
\\
  &&: q^{\sum_{1\le i<k\le r}\La(L_i,L_k)}
L_1\conv \cdots \conv L_r \longrightarrow L_r \conv \cdots \conv  L_1
\eneqn 
does not vanish.
\end{definition}

By applying \cite[Proposition 3.2.8]{KKKO18} inductively, we obtain the following lemma.
\Lemma
If $(L_1,\ldots,L_r)$ is a normal sequence of real simple modules, then 
the image of $\rmat{L_1,\ldots,L_r}$ is simple and coincides with
the head of $L_1\conv \cdots \conv L_r$ 
and also with the socle of $ L_r \conv \cdots \conv  L_1$,
up to grading shifts. 
\enlemma

\begin{lemma} \label{lemma:normal1}
A sequence $(L_1,\ldots,L_r)$ of real simple  modules in $R \gmod$ is a normal sequence if and only if 
$(L_2,\ldots,L_r)$ is a normal sequence and 
$$\La(L_1, \hd(L_2\conv\cdots \conv L_r)) = \sum_{2\le j\le r} \La(L_1,L_j).$$
\end{lemma}
\begin{proof}
 Set $\rmat{}'=\rmat{L_2,\ldots, L_r}
\cl L_2\conv \cdots \conv L_r \longrightarrow L_r \conv \cdots \conv  L_2$. 
In the diagram below,  the outer square  is commutative, because the both compositions of horizontal and vertical homomorphisms are equal to  $\rmat{L_1,\ldots,L_r}$. 
 Note that the vertical homomorphisms are non-zero and are equal to $\rmat{L_1,L_2\circ\cdots \circ L_r}$ and $\rmat{L_1,L_r\circ\cdots \circ L_2}$, respectively, by \cite[Lemma 3.1.5 (ii)]{KKKO18}.
 Let us denote by  $\phi$ the induced homomorphism between the images of horizontal homomorphisms.
$$\xymatrix@C=12ex@R=3ex
{L_1\conv (L_2\conv \cdots \conv L_r) 
\ar[rr]^{L_1 \circ \rmat{}'}\ar@{->}[ddd]|
{
\rmat{L_1,L_r}\circ \cdots \circ \rmat{L_1,L_2}\ake[3ex]}
\ar@{->>}[dr] 
& &L_1 \conv (L_r\conv \cdots \conv L_2)
\ar@{->}[ddd]|{
\rmat{L_1,L_2}\circ \cdots \circ \rmat{L_1,L_r}\ake[3ex]}
\\
&L_1 \conv \Im \rmat{}' \ar@{^{(}->}[ur] \ar@{->}[d]^\phi & \\
&\Im \rmat{}' \conv L_1  \ar@{^{(}->}[dr]& \\
(L_2\conv \cdots \conv L_r)\conv L_1  \ar@{->>}[ur] \ar@{->}
[rr]_{\rmat{}'\circ L_1}
&&(L_r\conv \cdots \conv L_2) \conv L_1
}
$$
Since $\Im(\rmat{L_1,\ldots,L_r})=\phi(L_1\conv \Im \rmat{}')$, we have that  
$\rmat{ L_1,\ldots,L_r}\neq 0$ if and only if $\rmat{}'\neq 0$ and $\phi\neq 0$. 
Under the condition that $\rmat{}'\neq 0$ (and hence $\Im \rmat{}' = \hd(L_2\conv \cdots \conv L_r)$ is simple), 
we have that  $\phi \neq 0$ if and only if $\phi=\rmat{L_1, \Im \rmat{}'}$, which is equivalent to the condition that $\La(L_1,L_2\conv \cdots \conv L_r) = \La(L_1,  \hd(L_2\conv \cdots \conv L_r))$ by \cite[Proposition 3.2.8]{KKKO18}. It completes the proof.
\end{proof}
In a similar way, we  have the following lemma.
\begin{lemma} \label{lemma:normal2}
A sequence $(L_1,\ldots,L_r)$ of real simple modules in $R \gmod$ is a normal sequence if and only if 
$(L_1,\ldots,L_{r-1})$ is a normal sequence and 
$$\La(\hd(L_1\conv\cdots \conv L_{r-1}), L_r) = \sum_{1\le j\le r-1} \La(L_j,L_r). $$
\end{lemma}

\Cor\label{cor:normal}
Let $(L_1,\ldots,L_r)$ be a sequence of real simple modules in $R \gmod$.
If $(L_1,\ldots,L_{r-1})$ is a normal sequence, then the following statements hold:
\bnum
\item If $\tL(L_k,L_r)=0$ for $1\le k\le r-1$, then
$(L_1,\ldots,L_r)$ is a normal sequence.
\item If $W^*(L_k)\cap W(L_r)\subset\{0\}$ for $1\le k\le r-1$, then
$(L_1,\ldots,L_r)$ is a normal sequence.
\item If $L_k$ and $L_r$ strongly commute for $1\le k\le r-1$, then
$(L_1,\ldots,L_r)$ is a normal sequence.
\ee
Similarly, if $(L_2,\ldots,L_{r})$ is a normal sequence,  then the following statements hold:
\bnum
\item If $\tL(L_1,L_k)=0$ for $2\le k\le r$, then
$(L_1,\ldots,L_r)$ is a normal sequence.
\item
If $W^*(L_1)\cap W(L_k)\subset\{0\}$ for $2\le k\le r$, then
$(L_1,\ldots,L_r)$ is a normal sequence.
\item
If $L_1$ and $L_k$ strongly commute for $2\le k\le r$, then
$(L_1,\ldots,L_r)$ is a normal sequence.
\ee
\encor

\Proof
We shall prove only the first set of statements.

\noi
(i)
We have 
$$\La(\hd(L_1\conv\cdots\conv L_{r-1}), L_r)\le
\sum_{k=1}^{r-1}\La(L_k, L_r)$$
by  \cite[Proposition 3.2.8]{KKKO18}. It implies
\eqn
&&0\le\tLa(\hd(L_1\conv\cdots\conv L_{r-1}), L_r)\le
\sum_{k=1}^{r-1}\tLa(L_k, L_r)=0,
\eneqn
where the first inequality follows from \eqref{eq:Lapos}. 
Hence we obtain
 $\tLa(\hd(L_1\conv\cdots\conv L_{r-1}), L_r)=
\sum_{k=1}^{r-1}\tLa(L_k, L_r)$, which implies
 $\La(\hd(L_1\conv\cdots\conv L_{r-1}), L_r)=
\sum_{k=1}^{r-1}\La(L_k, L_r)$.

\smallskip\noi(ii) We have
$\tLa(L_k,L_r)=0$ for $1\le k<r$ 
 by \cite[Proposition 2.12]{KKOP18}. Hence (ii) follows from (i).

\smallskip\noi
(iii) follows from  \cite[Proposition 3.2.13]{KKKO18}. 
\QED

\Cor
A sequence $(L_1,\ldots,L_r)$ of real simple modules in $R \gmod$ is a normal sequence if $\tL(L_i,L_j)=0$ for any $i,j$ such that $1\le i<j\le r$.
\encor

The following proposition is a consequence of 
Corollary~\ref{cor:normal}.
\begin{prop} \label{prop:normal sequence}
For any $\bolda = (\bolda_i)_{1 \le i \le l}, \boldb=(\boldb_i)_{1 \le i \le l} \in \Z_{\ge 0}^l$, the sequence 
\eqn 
(S_l^{\circ \bolda_l}, S_{l-1}^{\circ \bolda_{l-1}}, \ldots, S_1^{\circ \bolda_1}, M_1^{\circ \boldb_1}, M_2^{\circ \boldb_2}, \ldots, M_l^{\circ \boldb_l} )
\eneqn
is a normal sequence.
\end{prop}
\begin{proof}
 Note that the case $l=1$ is obvious.
Hence we assume $l>1$ and proceeds by induction on $l$. 
Then, we may assume that $(S_{l-1}^{\circ \bolda_{l-1}} ,  \ldots ,  S_1^{\circ \bolda_1},  M_1^{\circ \boldb_1},  \ldots ,  M_{l-1}^{\circ \boldb_{l-1}})$ is a normal sequence.

Since $S_k$ is a $\preceq$-cuspidal $R(\beta_k)$-module for $1\le k \le l$, we have, in particular 
$${\mathrm W}^*(S_l^{\circ \bolda_l}) \subset {\rm span}_{{\mathbb R}_{\ge 0}}\set{\gamma \in \Delta_+}{\gamma \succeq \beta_l}\ \text{(see \cite{KKOP18}).}   $$ 

On the other hand, 
since  $M_k$ is a quotient of a convolution product of $S_j$'s with $j \le k$, we have
\eqn
{\rm W} (S_k), \ {\rm W} (M_k)   \ \subset  {\rm span}_{{\mathbb R}_{\ge 0}}\set{\gamma \in \Delta_+}{\gamma \preceq \beta_{k}\prec \beta_{l}} \quad \text{for} \ 1 \le k \le l-1.
\eneqn
Hence, 
$$W^*(S_l^{\circ \bolda_l})\cap W(S_k^{\circ \bolda_k})=\{0\}\quad\text{and}\quad
W^*(S_l^{\circ \bolda_l})\cap W(M_k^{\circ \boldb_k})=\{0\}. $$
for $1\le k\le l-1$, 
Thus by Corollary~\ref{cor:normal} implies that
$$
(S_l^{\circ \bolda_l}, S_{l-1}^{\circ \bolda_{l-1}}, \ldots, S_1^{\circ \bolda_1}, M_1^{\circ \boldb_1}, M_2^{\circ \boldb_2}, \ldots, M_{l-1}^{\circ \boldb_{l-1}} )
$$
is a normal sequence.

\smallskip

Finally,
since $M_l$ strongly commutes with $S_k^{\circ \bolda_k}$ and $M_k^{\circ \boldb_k}$ for $1\le k \le l$  by \cite[Theorem 6.25]{Kimu12},
the assertion follows  again from Corollary~\ref{cor:normal}.
\end{proof}

The following proposition is a direct consequence of Proposition \ref{prop:normal sequence} and Lemma \ref{lemma:normal2}.
\begin{prop}
For $1\le k\le l$ and $\bolda \in \Z^\K_{\ge 0}$, we have
\eqn
\La(\hd(S_l^{\circ \bolda_l} \conv \cdots \conv S_1^{\circ \bolda_1}) , M_k) =\sum_{1\le j \le l} \bolda_j \La(S_j,M_k).
\eneqn
\end{prop}

\section{Triangular basis}\label{sec:tri}

\subsection{Localization and duals of simple objects} \label{subsec:Loc}

Let $w$ be an element  of  the Weyl group $W$.
In this subsection, we recall the localization of the monoidal category $\mathcal C_w$ via a real commuting family  of central objects following \cite{Local}.
Set $C_i:=M(w \varpi_i, \varpi_i)$.
Then,  $\set{[C_i]}{i \in I,  C_i\not\simeq\one}$ is the set of frozen variables of $\Anw$ up to a power of $q$.
 For each $\mu = \sum \mu_i \varpi_i \in \wl^+$, we denote by $C_\mu$ the  self-dual convolution product  $q^c \conv_{i\in I} C_{i}^{\circ \mu_i}$ for some $c \in \Z$.
\ Then we have  $C_\mu\simeq M(w\mu, \mu)$.
Note that 
$$C_\mu\conv C_{\mu'}\simeq q^{\la(\mu,\mu')}C_{\mu+\mu'}$$ holds
for any $\mu,\mu'\in\wl^+$. Here $\la$ is the bilinear form on $\wl$ defined by
$\la(\mu,\mu')=(\mu,w\mu'-\mu')$. 

 In \cite{Local}, it is proved that
there exist  a monoidal category $\Cwt = \mathcal C_w [C_i^{\circ -1}; i\in I]$   and a monoidal functor $\Phi: \mathcal C_w \to \Cwt$ satisfying 
 the following conditions. 
\begin{enumerate}
\item The objects $\Phi(C_i)$ are invertible in $\Cwt$; i.e., 
there exists an object $\Phi(C_i)^{-1}$ in $\tCw$ such that
$\Phi(C_i)\conv\Phi(C_i)^{-1}\simeq\one$ and $\Phi(C_i)^{-1}\conv\Phi(C_i)\simeq \one$. 
\item  For any monoidal functor $\Psi : \mathcal C_w \to \mathcal T$ to another monoidal category $\mathcal T$ in which $\Psi(C_i)$ is
invertible for every $i$, there exists a monoidal functor $\Psi'
\cl\tCw\to\mathcal T$ such that $\Psi \simeq \Psi' \circ \Phi$. Moreover $\Psi'$ is unique up to an isomorphism.
\setcounter{myc}{\value{enumi}}
\ee
 Then there exists a real commuting family $\{\wdt C_\mu\}_{\mu\in\wl}$ in $\tCw$ 
such that
$\wdt C_\mu=\Phi(C_\mu)$ for every $\mu\in \wl^+$ and
$\wdt C_\mu\conv \wdt C_{\mu'}\simeq  q^{\la(\mu,\mu')}\wdt C_{\mu+\mu'}$
for every $\mu,\mu'\in \wl$.

The localization $\tCw$ satisfies also the following properties.

\be\setcounter{enumi}{\value{myc}}
 \item 
Every simple object of $\Cwt$ is isomorphic to $\Phi(S) \circ \wdt C_\mu$ for some simple object $S$ of $\mathcal C_w$ and $\mu\in \wl$.
\item
For two simple objects $S$ and $S'$ in $\mathcal C_w$,
$\Phi(S)\conv \wdt C_\mu\simeq \Phi(S') \conv\wdt  C_{\mu'}$ in $\Cwt$
if and only if 
$ S\conv C_{\mu+\la}\simeq   S' \conv C_{\mu'+\la}$ 
in 
$\mathcal C_w$
for some $\la\in\wl$
such that $\mu+\la$, $\mu'+\la\in \wl^+$.
Here we ignore grading shifts.
\end{enumerate}

 Note that the category $\Cwt$ is abelian and every 
 object  has 
 finite length.
 Moreover $\Phi\cl \Cw\to\tCw$ is exact. 
  The grading shift functor $q$ and the contravariant functor $M \mapsto M^*$ on $\mathcal C_w$   are extended to $\Cwt$. 
Moreover  they satisfy 
$$(M\conv N)^*\simeq q^{(\wt(M),\wt(N))}N^*\conv M^*$$
 for $M,N\in\tCw$.
Here, for $M\simeq\Phi(L)\conv\wdt C_\la\in\tCw$
with $L\in\Cw$ and $\la\in\wl$, we set
 $\wt(M)=\wt(L)+(w\la-\la)$ .

For any simple  object  $M\in\tCw$, there exists a unique $n\in\Z$  such that $(q^nM)^* \simeq q^n M$.
We say that $q^n M$ is self-dual in this case. 

We denote by $\Irrsd(\tCw)$ the set of the isomorphism classes of self-dual simple objects of $\tCw$.

The Grothendieck ring $K(\Cwt)$ of $\Cwt$ is a  $\Z[q^{\pm1}]$-algebra with a basis consisting of the class of self-dual simple objects.  
 Moreover it  is  isomorphic to 
the right ring of quotients of the ring $K(\mathcal C_w)$ with respect to the  multiplicative subset 
$$S:=\set{q^k\prod_{i\in I}[C_i]^{a_i}}{k\in \Z, \ (a_i)_{i \in I} \in \Z_{\ge 0}^{I} }.$$
Hence the ring $\Z[q^{\pm1/2}] \otimes_{\Z[q^{\pm1}]}K(\Cwt)$, which is a localization of $\Z[q^{\pm1/2}] \otimes_{\Z[q^{\pm1}]}K(\mathcal C_w)$, 
is the $\Z[q^{\pm1/2}]$-subalgebra of the skew field of fractions $K$ of an initial quantum torus $ \Z[q^{\pm1/2}]  [X_i^{\pm1} \ |   \ i \in \K]$ generated by all the exchangeable cluster variables,  the frozen variables, and the inverses of the frozen variables. Thus it is the quantum cluster algebra in the sense of \cite{BZ05}.

Recall that a pair   $(\eps\cl X \otimes  Y \rightarrow \one ,\;  \eta\cl \one \rightarrow Y \otimes X)$ of morphisms in a monoidal category  with a unit object $\one$
is called an \emph{adjunction} if 
\begin{enumerate}
\item[(a)] $X \simeq X \otimes \one \To[{\;X \otimes \eta\;}] X \otimes Y \otimes X \To[{\;\eps \otimes X\;}] \one \otimes X \simeq X$
is the identity of $X$, and
\item[(b)]  $Y \simeq \one \otimes Y \To[{\;\eta \otimes Y\;}]
Y \otimes X \otimes Y \To[{\;Y \otimes \eps\;}]Y \otimes \one \simeq Y$
is the identity of $Y$.
\end{enumerate} 
In this case, the pair $(X,Y)$ is called a  \emph{dual pair}. 
When $(\eps, \eta)$ is an adjunction, we say that $X$ is a \emph{left dual} of 
$Y$ and $Y$ is a \emph{right dual} of  $X$. 
Note that a left dual (resp.\ right dual) of an object is unique up to a unique isomorphism if it exists.

 The following theorem plays an important role in this paper. 
\begin{theorem}[{\cite{Local}}] \label{thm:dual}
Every simple object $M$ in $\Cwt$ has a right dual and a left dual.
\end{theorem}

We denote by $\mathscr D(M)$  the right dual of a simple object $M$ in $\Cwt$ and 
 by $\mathscr D^{-1}(M)$  the left dual of  $M$.

\subsection{Cluster variables and exchange matrices} \label{subsec:cvarex}
Set $$I'\seteq  \set{i_k}{ k_+=l+1 }=\set{i_k}{k\in \Kfr} \subset I,$$ 
 where $k_+:=\min(\set{j}{k<j\le l,\;i_j=i_k} \cup \{l+1\})$. 
Let $\set{M_j}{j\in\K}$ be a   monoidal cluster   in $\mathcal C_w$.
 Recall that $l=\sharp\K$ and $n=\sharp\Kfr=\sharp I'$.
Taking a total order on $\K$ such that $j <j'$ for all $j \in \Kex$ and $j' \in \Kfr$,
we shall identify $\K$, $\Kex$ and $\Kfr$ with
$\{1,\ldots,l\}$, $\{1,\ldots,l-n\}$ and $\{l-n+1,\ldots,l\}$, respectively. 

Then the matrices $\Lambda=(\La(M_i,M_j))_{1\le i,j\le l}$ and 
the exchange matrix $\tB$ satisfy the following relation
(see \cite{BZ05}):
$$\Lambda \tB = 
\begin{bmatrix}-2 \id_{(l-n)\times(l-n)}  \\
0
\end{bmatrix}.$$

Let $D = (d_{t,j})_{t\in I', j \in \K }$ be an $n \times l$ matrix defined by
\eq \label{eq:weightmatrix}
\sum_{t\in I'}d_{t,j} \al_t = \wt(M_j) \quad \text{for all} \  j \in \K,\eneq
so that $\sum_{t \in I'} (D\boldb)_t \, \al_t = \wt(M(\boldb))$ for $\boldb \in \Z_{\ge 0}^J$.
Then we have 
$$
D \tB =0.
$$

Now let $\Lambda'$ be the $(l-n) \times l$-matrix obtained by taking the first $(l-n)$ rows of $\La$.
Then we have 
$$
\begin{bmatrix}
\La' \\
D
\end{bmatrix} \tB = 
\begin{bmatrix}-2 \id_{(l-n)\times(l-n)}  \\
0
\end{bmatrix}.$$

\begin{lemma} \label{lem:invertible}
The matrix $\begin{bmatrix}
\La' \\
D
\end{bmatrix}  $ is an invertible,  and hence $\tB =   \begin{bmatrix}
\La' \\
D
\end{bmatrix}^{-1}
 \begin{bmatrix}-2 \id_{(l-n)\times(l-n)}  \\
0
\end{bmatrix}.$ 

\end{lemma}
\begin{proof}
Set $\Ker \tB := \set{v \in \Q^l }{v \tB=0}$.  Then the row vectors of $\La'$ form a $\Q$-linearly independent  set of cardinality $l-n$, and  the $\Q$-vector subspace of $\Q^l$ spanned by them meets $\Ker \tB$ trivially. 
Since the row vectors of  $D$  belong to $\Ker \tB$, it is enough to show that the column rank of $D$ is equal to $n$.
By the definition of $D$, it is equivalent to saying  that the set $\set{\wt (M_j)}{j \in  J} $ generates the space $\soplus_{i\in I'} \Q\al_i$.
It can be checked easily in the case of the 
initial seed attached to a reduced expression of $w$.  Since the subspace
$$\Span_\Q \set{\wt (M_j)}{j \in  J} \subset \soplus_{i\in I'} \Q\al_i$$
is invariant under any mutation, we conclude that $\set{\wt (M_j)}{j \in  J} $ generates $\soplus_{i\in I'} \Q\al_i$ for any seed.
\end{proof}

\subsection{Dominance order}
Let $ \mcluster =\set{M_j}{j \in \K}$ be a  monoidal cluster   in $\mathcal C_w$and $\shs$ the monoidal seed determined by $\mcluster$.  
We define the order 
 $\succeq_{\shs}$ on $\Z_{\ge 0}^\K$ as follows:
$$
\boldb \succeq_{\shs} \boldb' 
$$
if and only if
\begin{enumerate}
\item $ \wt(M(\boldb)) = \wt(M(\boldb'))$, 
\item $ \La(M(\boldb), M_k) \ge \Lambda(M(\boldb'),M_k) \ \text{for all} \ k \in \Kex$.
\end{enumerate}
Note that   condition (1) implies that 
$ \La(M(\boldb), M_k) = \Lambda(M(\boldb'),M_k) \ \text{for all} \ k \in \Kfr$, because $\La(M(\boldb),M_k)=(w \varpi_{i_k}+\varpi_{i_k}, \wt(M(\boldb)) )$ 
depends only on $\wt(M(\boldb))$ (see, \cite[Theorem 4.12]{KKOP18}).  

Note that $\boldb +\bolda \succeq_{\shs} \boldb'+\bolda$ for any $\boldb, \boldb',\bolda \in \Z_{\ge 0}^\K$ with $\boldb \succeq_{\shs} \boldb'$.
Hence we can extend the relation to the one on $\Z^\K$ by
$$\text{$\boldb \succeq_{\shs} \boldb'$ if $\boldb +\bolda \succeq_{\shs} \boldb'+\bolda$ for some $\bolda \in \Z_{\ge0}^J$ 
such that $\boldb+\bolda,\;\boldb'+\bolda\in\Z_{\ge0}^\K$. }$$

Recall that $\tLa(X,Y)=\dfrac{1}{2}(\La(X,Y)+(\wt(X),\wt(Y))) \in \Z$  for $X,Y \in R \gmod$. 
Hence $\boldb\succeq_{\shs} \boldb'$ implies that $\La(M(\boldb),M_k)-\La( M(\boldb'),M_k) \in 2 \Z_{\ge 0}$ for all $k\in \Kex$.

Thus we have
$$\boldb \succeq_{\shs} \boldb'  \qquad \text{if and only if} \qquad D(\boldb-\boldb')=0 \ \text{and} \ -\La'(\boldb-\boldb') \in 2 \Z_{\ge 0}^{\Kex},$$
where $D$  and $\La'$  are  the matrices given in subsection \ref{subsec:cvarex}.

\begin{prop}
 We have 
$$\boldb \succeq_{\shs} \boldb'  \qquad \text{if and only if} \qquad  \boldb-\boldb' = \tB v \ \text{for some } \ v \in \Z_{\ge 0}^{\Kex}. $$
In particular, the relation $\succeq_{\shs}$ is an order on $\Z^{\K}$.
\end{prop}
\begin{proof}
We have
\eqn 
\boldb \succeq_{\shs} \boldb'
\qquad \text{if and only if} \qquad 
\begin{bmatrix}
\La' \\ D
\end{bmatrix} ( \boldb-\boldb') =\begin{bmatrix}
-2v \\0
\end{bmatrix}
\quad
\text{for some}  \ v  \in \Z_{\ge 0}^{\Kex}.
\eneqn
By Lemma \ref{lem:invertible}, it is equivalent to saying that
\eqn
\boldb-\boldb' = \tB v \
 \text{for some} \ v  \in \Z_{\ge 0}^{\Kex},
\eneqn
as desired. 
\end{proof}

\begin{remark}
 The order $\succeq_{\shs}$ is the same as the one in \cite[Definition 3.1.1]{Qin17}.  
\end{remark}

\subsection{Heads and Socles}

Let $ \mcluster  =\set{M_j}{j \in \K }$ be a monoidal cluster in $\mathcal C_w$, 
 $\seed$ the monoidal seed determined by $\mcluster$.  
Then any simple module $X \in \mathcal C_w$
satisfies the following relation in the Grothendieck ring $K(\mathcal C_w)$
 by Lemma~\ref{lem:LP}: 
\eq 
&&[X \conv  M(\bolda)] =\sum_{k\in K}q^{c_k} [M(\boldb(k))]
\label{eq:Laurent}
\eneq
for a family $\{\boldb(k)\}_{k\in K}$ in $\Z_{\ge 0}^{\K}$,
a family $\{c_k\}_{k\in K}$ of integers and $\bolda \in \Z_{\ge 0}^{\K}$
such that
$\bolda_j=0$ for $j \in \Kfr$.

\Lemma Let $X$ be a simple module in $\Cw$, and
assume the relation \eqref{eq:Laurent}.
Then we have
\eq 
&&[M(\bolda)\conv X] =\sum_{k\in K}q^{c'_k} [M(\boldb(k))]
\eneq
where  $c'_k=c_k+\La(M(\boldb(k)),M(\bolda))$. 
\enlemma
\Proof
We have
\eqn
[M(\bolda)]\cdot[X]\cdot[M(\bolda)]
&&=\sum_{k\in K}q^{c_k} [M(\bolda)]\cdot[M(\boldb(k))]\\*
&&=\sum_{k\in K}q^{c_k}q^{\La(M(\boldb(k)),M(\bolda))} 
[M(\boldb(k))]\cdot[M(\bolda)].
\eneqn
Hence we obtain the desired result.
\QED

 Recall that $X \conv M(\bolda)$ has a simple head $X\hconv M(\bolda)$
and a simple socle $X\sconv M(\bolda)$, since $M(\bolda)$ is a real simple module.

\begin{lemma}\label{lem:Mab}
Let $X$ be a simple module in $\Cw$, and
assume  relation \eqref{eq:Laurent}.
Then  we have the following properties.
\bnum
\item There exists a unique $k_0\in K$ such that
$X\hconv M(\bolda)\simeq q^{c_{k_0}}M(\boldb(k_0))$.
Moreover we have
$\boldb(k_0) \succ_{\shs}\boldb(k)$ for any $k\in K\setminus\{k_0\}$.
\item There exists a unique $k_1\in K$ such that
$X\sconv M(\bolda)\simeq q^{c_{k_1}}M(\boldb(k_1))$.
Moreover we have
$\boldb(k) \succ_{\shs}\boldb(k_1)$ for any  $k\in K\setminus\{k_1\}$.
\item If $k_0=k_1$, then $K=\{k_0\}$.
\item If ${\rm char} (\cor)=0$, then
$c_{k_0}< c_k$ for  any  $k\in K\setminus\{k_0\}$ and
$c_k <c_{k_1}$ for any  $k\in K\setminus\{k_1\}$.
\end{enumerate}
\end{lemma}
\begin{proof}
The existence of $k_0$ and $k_1$ is obvious.

For an arbitrary $\boldc \in \Z_{\ge 0}^{\K}$,  by multiplying $M(\boldc)$  to \eqref{eq:Laurent},  we obtain that 
\eqn
&&q^{-\tLa(M(\bolda), M(\boldc))} [X \conv M(\bolda+\boldc)]% \\
=\sum_{k\in K}q^{c_k-\tLa(M(\boldb_j), M(\boldc))}  [M(\boldb(k)+\boldc)].
\eneqn
Since $M(\bolda+\boldc)$ is a real simple module, $X \conv M(\bolda+\boldc)$ has a simple head and a simple socle. 
Because  the functor $-\conv M(\boldc)$ is exact, the module $q^{c_{k_0}-\tLa(M(\boldb(k_0)), M(\boldc))} M(\boldb(k_0)+\boldc)$  
is a quotient of $q^{-\tLa(M(\bolda), M(\boldc))} X\conv M(\bolda+\boldc)$ and hence is  equal to 
the head of $q^{-\tLa(M(\bolda), M(\boldc))} X\conv M(\bolda+\boldc)$. 

%%%%%%%%%%%%%%%%

Then it follows by  \cite[Theorem 4.1.1]{KKKO18} 
that
$$\La(M(\boldb(k_0)+\boldc), M(\bolda+\boldc))  
\ge   \La(M(\boldb(k)+\boldc), M(\bolda+\boldc))\qt{for any $k\in K$,} $$
and hence 
$$\La(M(\boldb(k_0)), M(\boldc))- \La(M(\boldb(k)), M(\boldc)) 
\ge  \La(M(\boldb(k),M(\bolda)) -\La(M(\boldb(k_0)),M(\bolda))).$$

Because $\boldc$ is arbitrary, we may replace $\boldc$ by $t\, \boldc$ to obtain
$$t \bl\La(M(\boldb(k_0)), M(\boldc))- \La(M(\boldb(k)), M(\boldc))\br \ge
\La(M(\boldb(k),M(\bolda)) -\La(M(\boldb(k_0)),M(\bolda)))$$
for any  $\boldc \in \Z_{\ge 0}^{\K}, \ t\in \Z_{\ge 1}$.
It follows that  $$\La(M(\boldb(k_0)), M(\boldc))\ge \La(M(\boldb(k)), M(\boldc)) \quad \text{for any}  \ \boldc \in \Z_{\ge 0}^{\K}.
$$

On the other hand, we  know that 
$\La(M(\boldb(k_0)),M(\bolda)) > \La(M(\boldb(k)),M(\bolda))$ for any
$k\in K\setminus\{k_0\}$ by \cite[Corollary 4.1.2]{KKKO18}.
Hence we obtain $\boldb(k_0) \succ_{\shs}\boldb(k)$ 
for $k\in K\setminus\{k_0\}$.

\smallskip
\noi
(ii) can be proved similarly.

\smallskip
\noi
(iii) follows from (i) and (ii).

\smallskip\noi
(iv) is Theorem 4.2.1 and Corollary 4.2.2 in \cite{KKKO18}.
\end{proof}

\Lemma\label{lem:expRL}
Let $X$ be a simple module in $\Cw$.
Then, we have
\bnum
\item There exists $\bolda$, $\boldb\in\Z_{\ge0}^\K$ such that $\bolda_k=0$ for $k\in\Kfr$ and $X\hconv M(\bolda)\simeq M(\boldb)$ up to a grading shift.
\item If 
$\bolda$, $\bolda'$, $\boldb$, $\boldb'\in\Z_{\ge0}^\K$ 
satisfy
$X\hconv M(\bolda)\simeq M(\boldb)$ and $X\hconv M(\bolda')\simeq M(\boldb')$
up to grading shifts, then one has
$\boldb-\bolda=\boldb'-\bolda'$.
\ee
The similar statements hold for $\sconv$.
\enlemma
\Proof
Since the statements for $\sconv$ can be proved similarly to the one for $\hconv$,
we only give the proof of the statement
on $\hconv$.

\smallskip\noi
(i) follows from Lemma~\ref{lem:Mab}.

\smallskip\noi
(ii) we have
\eqn
M(\bolda'+\boldb) \simeq M(\boldb) \hconv M(\bolda') 
&&\simeq (X \hconv M(\bolda)) \hconv M(\bolda') \\
&&\simeq \hd(X\conv M(\bolda') \conv M(\bolda)) 
\simeq M(\boldb'+\bolda),
\eneqn
so that we have $\bolda'+\boldb=\boldb'+\bolda$, as desired.
\QED

\begin{definition}
For any simple $X$ in $\mathcal C_w$  and any   monoidal cluster   $ \mcluster=\set{M_j}{j\in \K}$, take $\bolda$, $\bolda'$, $\boldb$, $\boldb'\in\Z_{\ge0}^\K$
such that $X\hconv M(\bolda)\simeq M(\boldb)$ and $X\sconv M(\bolda')
\simeq M(\boldb')$ up to a grading shift.
Then, we set
$$\expR(X):=\boldb-\bolda, \qquad  \text{and} \qquad \expL(X):=\boldb'-\bolda'.$$
These can be extended to $\expR$ and $\expL$  from the set 
$\Irrsd(\tCw)$
of the isomorphism classes of  self-dual  simples in $\Cwt$ 
to the set $\Z^\K$ by 
$$\expR(X\conv M(\boldc)) :=  \expR(X)+ \boldc, \quad \text{and} \quad
\expL(X\conv M(\boldc)) :=  \expL(X)+ \boldc
$$
for any $\boldc \in \Z^\K$ with $\boldc_j=0$ for $j \in \Kex$.
\end{definition}

Note that $\expR$ and $\expL$ are
 well-defined  by Lemma~\ref{lem:expRL}.

\begin{remark}\bnum
\item 
The map $\bold{deg}$ in \cite[Definition 3.1.4]{Qin17} corresponds to the map  $\expR$. 
\item
Since $X\sconv M(\bolda')\simeq M(\bolda')\hconv X$ up to a grading shift,
 $\expL$  is as important as  $\expR$. 
\ee
\end{remark}

\Lemma
Let $X$ be a simple module in $\mathcal C_w$.
If $\bolda, \boldb \in \Z_{\ge 0}^\K$ satisfy
$\boldb-\bolda =
\expR(X)$ \ro
respectively, $\boldb-\bolda =\expL(X)$\rf,
then we have
$$X\hconv M(\bolda)\simeq M(\boldb)
\qt{\ro respectively,  $M(\bolda)\hconv X\simeq M(\boldb )$\rf}$$
up to a grading shift.
\enlemma

\begin{proof}
Let us take $\bolda'$ and $\boldb'$,
such that $X \hconv  M(\bolda')\simeq M(\boldb')$.
(We omit the grading shifts.) 
Then one has $\boldb-\bolda=\expR(X)=\boldb'-\bolda'$.
Hence we have a sequence of morphisms whose composition is an epimorphism:
\eqn
&&X\conv M(\bolda)\conv M(\boldb')\simeq
X\conv M(\bolda+\boldb')
=X\conv M(\bolda'+\boldb)
\simeq X\conv M(\bolda') \conv M(\boldb) \\
&&\hs{40ex}\epito (X \hconv M(\bolda'))\conv M(\boldb) 
\simeq M(\boldb')\conv M(\boldb).
\eneqn
On the other hand, $X\conv M(\bolda)\conv M(\boldb')$
has a simple head.
Hence we obtain
$$\bl X\hconv M(\bolda)\br\hconv M(\boldb')\simeq  M(\boldb')\conv M(\boldb)
\simeq M(\boldb)\hconv M(\boldb').$$
Then \cite[Corollary 3.7]{KKKO15} implies
$X\hconv M(\bolda)\simeq M(\boldb)$. The proof for $\expL[]$ is similar. 
\end{proof}

 The following lemma is obvious.
\Lemma
For  any  simple module $X$ in $\Cw$ and  any  $\bolda\in\Z_{\ge0}^\K$,
we have
$$\expR(X\hconv M(\bolda))=\expR(X)+\bolda,\quad
\expL( M(\bolda)\hconv X)=\bolda+\expL(X).$$
\enlemma

\begin{lemma}
The maps 
$\expR$ and  $\expL$ are injective.
\end{lemma}
\begin{proof}
Let $X, Y$ be self-dual  simples in $\mathcal C_w$ such that
$\expR(X)=\expR(Y)$.
Take $\bolda, \boldb\in \Z_{\ge 0}^J$ such that $\expR(X)=\boldb-\bolda$.
Then we have
$X \hconv M(\bolda) \simeq M(\boldb)$ and 
$Y \hconv M(\bolda) \simeq M(\boldb)$ up to a grading shift. 
We conclude that $X \simeq Y$    by \cite[Corollary 3.7]{KKKO15} and hence $\expR$ is injective. Similarly one can show that $\expL$ is injective.
\end{proof}

These maps are closely related with the duality.
\begin{lemma} \label{lem:expanddual}
Let $(L, R)$  be a dual pair of simples in $\Cwt$. 
Then we have
\eqn 
\expR(L)+\expL(R)  =0.
\eneqn
\end{lemma}
\begin{proof}
 Write $L \hconv M(\bolda) \simeq  M(\boldb) $ 
for $\bolda$, $\boldb\in \Z_{\ge0}^\K$ 
so that $\expR(L) =\boldb-\bolda$.
Since $R$ and $L$ are simple in $\Cwt$, there exist
$\boldc,\boldc' \in \Z_{\ge 0}^{\Kfr}$ such that
$M(\boldc) \conv R $
and $L\conv M(\boldc')$ are simple objects in  $\mathcal C_w$.

Then $\one\to R  \conv   L$ induces a monomorphism
\eq
&&M(\boldc)  \conv M(\boldc')
 \monoto \bl M(\boldc) \conv R\br \conv 
\bl L\conv M(\boldc')\br\qt{in $\tCw$,}\label{eq:ac1}\eneq
and $L\conv M(\bolda)\epito M(\boldb)$
 induces an epimorphism
\eq&&\bl L\conv M(\boldc')\br \conv M(\bolda)\simeq
L\conv M(\bolda)\conv M(\boldc')\epito  M(\boldb)\conv M(\boldc')
\qt{in $\tCw$.}\label{eq:ac2}
\eneq

From the definition of morphisms in $\Cwt$ (see \cite{Local}), 
for sufficiently large $\boldc$ and $\boldc'$, we may assume that
the morphism \eqref{eq:ac1} and the composition of morphisms in
\eqref{eq:ac2}
are morphisms in $\Cw$.
Then we have a chain of morphisms
\eqn \bl M(\boldc)  \conv M(\boldc')\br\conv M(\bolda)
&& \monoto \bl M(\boldc) \conv R\br \conv 
\bl L\conv M(\boldc')\br \conv M(\bolda)\\ 
&&\hs{5ex}\epito \bl M(\boldc) \conv R \br\conv \bl M(\boldb)\conv M(\boldc')\br \eneqn
in $\Cw$.
  By \cite[Lemma 3.1.5]{KKKO18},
the composition is non-zero. Since 
$\bl M(\boldc)  \conv M(\boldc')\br\conv M(\bolda)$ is a simple  module, 
it is a monomorphism.
Therefore, $\bl M(\boldc) \conv R \br\sconv
 \bl M(\boldb)\conv M(\boldc')\br\simeq 
\bl M(\boldc)  \conv M(\boldc')\br\conv M(\bolda)$.
Hence we obtain
$$\boldc+\expL(R)=\expL\bl M(\boldc) \conv R\br=(\boldc+\boldc'+\bolda)-(\boldb+\boldc')
=\boldc+\bolda-\boldb.$$
Thus $\expL(R)= \bolda - \boldb$, as desired. 
\end{proof}

\begin{prop} \label{prop:expRforcusp}
Let $ \mcluster_{\bold i}=\set{M_k}{1\le k\le l}$  and  $\seed_{\bold i}$   be the  monoidal cluster  and the monoidal seed   associated with a  reduced expression $\underline{w} = s_{i_1 }\cdots s_{i_l}$, where $\bold i = (i_1,\ldots, i_l)$   \ro see, subsection~\ref{subsec:initial seeds}\rf. 

For $\boldc \in \Z_{\ge 0}^{l}$, set
\eqn
P_{\bold i}(\boldc):=\hd(S_l^{\boldc_l} \conv \cdots \conv S_1^{\boldc_1}).
\eneqn
Then we have
\eqn
\expR[ \seed_{\bold i}](P_{\bold i}(\boldc)) = (\boldc_1-\boldc_{1_+}, \ldots, \boldc_l- \boldc_{l_+}),
 \eneqn
where
$k_+:=\min(\set{j}{k<j\le l,\;i_j=i_k} \cup \{l+1\})$, and 
$\boldc_{l+1}:=0$.
\end{prop}
\begin{proof}
By Proposition \ref{prop:normal sequence}, we know that 
\eqn 
(S_l^{\circ \boldc_l}, S_{l-1}^{\circ \boldc_{l-1}}, \ldots, S_1^{\circ \boldc_1}, M_1^{\circ {\boldc_1}_+}, M_2^{\circ {\boldc_2}_+}, \ldots, M_l^{\circ {\boldc_l}_+} )
\eneqn
is a normal sequence.

Set $\boldc_+ \seteq \sum_{k=1}^l {\boldc_k}_+ \bolde_k$,  where $\set{\bolde_j}{j \in \K}$ is the
 basis of  $\Z^J$ such that $\boldc=\sum_{j\in \K} \boldc_j \bolde_j$. 
Then we have $\boldc_+ = \sum_{j=1}^l \boldc_j \bolde_{j_-}$,  where $j_-=\max (\set{1\le k < j}{ i_k=i_j})\cup\{0\})$ and $\bolde_0:=0$.
Hence $M(\boldc_+) \simeq M_{1_-}^{\circ \boldc_1} \conv \cdots \conv M_{l_-}^{\circ \boldc_l} $.
Because we have
\eqn
S_k \hconv M_{k_-} \simeq M_k \quad \text{for} \ 1\le k\le l,
\eneqn
it follows that 
\eqn
&&P_{\bold i}(\boldc) \hconv M(\boldc_+)
\simeq\hd( S_l^{\circ \boldc_l} \conv S_{l-1}^{\circ \boldc_{l-1}} \conv \cdots\conv S_2^{\circ \boldc_2} \conv S_1^{\circ \boldc_1} \conv M(\boldc_+)) \\
&&\simeq\hd( S_l^{\circ \boldc_l} \conv S_{l-1}^{\circ \boldc_{l-1}} \conv \cdots\conv S_2^{\circ \boldc_2}\conv S_1^{\circ \boldc_1} \conv M_{1_-}^{\circ \boldc_1}\conv M_{2_-}^{\circ \boldc_2}\conv \cdots\conv M_{l_-}^{\circ \boldc_l} ) \\
&&\simeq\hd( (S_l^{\circ \boldc_l} \conv S_{l-1}^{\circ \boldc_{l-1}} \conv \cdots \conv  S_2^{\circ \boldc_2} )\conv (S_1^{\circ \boldc_1} \hconv M_{1_-}^{\circ \boldc_1})\conv (M_{2_-}^{\circ \boldc_2}\conv \cdots\conv M_{l_-}^{\circ \boldc_l} )) \\
&&\simeq\hd( (S_l^{\circ \boldc_l} \conv S_{l-1}^{\circ \boldc_{l-1}} \conv \cdots \conv  S_2^{\circ \boldc_2} )\conv M_1^{\circ \boldc_1} \conv (M_{2_-}^{\circ \boldc_2}\conv \cdots\conv M_{l_-}^{\circ \boldc_l} )) \\
&&\simeq\hd( (S_l^{\circ \boldc_l} \conv S_{l-1}^{\circ \boldc_{l-1}} \conv \cdots \conv  S_2^{\circ \boldc_2} ) \conv (M_{2_-}^{\circ \boldc_2}\conv \cdots\conv M_{l_-}^{\circ \boldc_l} ) \conv M_1^{\circ \boldc_1}) \\
&&\simeq \cdots \\
&&\simeq\hd(M_l^{\circ \boldc_l} \conv \cdots \conv M_1^{\circ \boldc_1})
\simeq M(\boldc),
\eneqn
as desired.
\end{proof}

\begin{corollary} \label{cor:surj1}
The map $\expR[\shs_{\bold i}] : \Irrsd(\Cwt) \to \Z^\K$ is  bijective.
\end{corollary}
\begin{proof}
Since $\expR[{\shs_{\bold i}}]$ is injective, we will show that it is surjective.

For each $1 \le k\le l+1 $ and $ m \ge 0$, define inductively $k_{(0)}:=k$, $k_{(m)}:=(k_{(m-1)})_+$. 
Let $m_k$ be the non-negative integer such that $k_{(m_k)}\le l$ and $(k_{( m_k+1)})_+=l+1$. 
For a given $\bolda \in \Z^{l}$, take an $\bolda' \in \Z^{l}$ such that 
\label{eq:akp}
 \eqn &&\bolda_k=\bolda'_k \quad \text{for $k \in \Kex$, and}  \\
&&\boldc_k\seteq \bolda'_k + \bolda'_{k_+} + \bolda'_{k_{(2)}}+\cdots 
+\bolda'_{k_{(m_k-1)}}+\bolda'_{k_{(m_k)}} \ge 0\quad 
\text{for $1 \le k \le l$.} \eneqn 
 Since $k_{(m_k)} \in \Kfr$, one can find such an $\bolda'$  
by taking sufficiently large $\bolda'_{k_{(m_k)}}$.
 Then we have
$\boldc_k-\boldc_{k_+}=\bolda'_k$. 
Hence, the above proposition implies that
$\expR[{\shs_{\bold i}}](P(\boldc)) = \bolda'$.
Thus, $\bolda-\bolda'\in\Z^{\Kfr}$ and 
$$\expR[{\shs_{\bold i}}](P(\boldc)\conv M(\bolda -\bolda')) = \bolda'+(\bolda-\bolda')=\bolda,$$
as desired.
\end{proof}

\subsection{Tropical transformations}
Let us fix a   monoidal cluster   $ \mcluster = \set{M_j}{ j \in \K}$ in $\mathcal C_w$ 
 and let $\seed$ be the associated monoidal seed. 
For $i \in \Kex$  and  $j\in \Kfr$, we set $b_{ij}:=-b_{ji}$.
Let us denote by  $\{\bolde_j\}_{j \in \K}$ the natural basis of $\Z^J$.

Fix $k \in \Kex$ and we denote by 
$\mu_k(\shs)=\set{M'_j}{ j \in \K}$ the mutation of $\shs$
 in the direction $k$.
Recall that  $M'_j=M_j$ for $j\not=k$ and 
there exists an exact sequence (ignoring the gradings)
$$0 \to U \to M_k \conv M_k' \to V \to 0,$$
where
\eqn
V = M\left (\sum_{b_{ki}> 0} b_{ki}\bolde_i\right), \quad \text{and} \quad U = M\left(\sum_{b_{ik} >0} b_{ik}\bolde_i\right).
\eneqn

Set $ [a]_+  :=\max\{a,0\}$ for $a \in\Z$.

For $\boldg \in \Z^{\K}$, we define
\eqn \pR_{\mu_k(\shs),\shs} (\boldg) :=\boldg', \quad \text{where}
\eneqn
\eqn
\boldg'_i = 
\begin{cases}
-\boldg_k &\text{if} \ i =k, \\
\boldg_i +[b_{ki}]_+ \boldg_k&\text{if} \ i \neq k, \ \boldg_k \ge 0, \\
\boldg_i +[b_{ik}]_+ \boldg_k&\text{if} \ i \neq k, \ \boldg_k \le 0,
\end{cases} 
\eneqn
 and
\eqn\pL_{\mu_k(\shs),\shs} (\boldg) := \boldg''   \quad \text{where} \eneqn
\eqn
\boldg''_i = 
\begin{cases}
-\boldg_k &\text{if} \ i =k, \\
\boldg_i +[b_{ik}]_+ \boldg_k&\text{if} \ i \neq k, \ \boldg_k \ge 0, \\
\boldg_i +[b_{ki}]_+ \boldg_k&\text{if} \ i \neq k, \ \boldg_k \le 0.
\end{cases} 
\eneqn

Setting $\mu_k(\shs)\seteq \shs'$, the maps
$\pR_{\shs',\shs}$ and $\pR_{\shs,\shs'}$
(respectively, $\pL_{\shs',\shs}$ and $\pL_{\shs,\shs'}$) 
are inverse to each other.

\begin{theorem}
For any simple $X \in \Cwt$, we have
\eqn
&\pR_{\mu_k(\shs),\shs} (\expR(X)) = \expR[{\mu_k(\shs)}](X) \quad \text{and} \quad
&\pL_{\mu_k(\shs),\shs} (\expL[{\shs}](X)) = 
\expL[{\mu_k(\shs)}](X).
\eneqn
\end{theorem}
\begin{proof}
Since the proofs are similar, we will show only the second equality.
In the course of the proof, we ignore the grading shifts. 

 Take $\bolda$, $\boldb\in \Z_{\ge0}^\K$
such that $\boldg\seteq\expL[{\shs}](X) =\boldb -\bolda$.
Then we have $X \sconv M(\bolda) = M(\boldb)$.
We may assume that either $\bolda_k=0$ or $\boldb_k=0$.

\medskip\noi
(Case $\bolda_k=0$)\hs{1.5ex} We have $\boldg_k=\boldb_k\ge0$.

Set $\bar \boldb = \boldb-\boldb_k\bolde_k$. Then we have
\eqn 
&&X \sconv M'(\bolda+\boldb_k\bolde_k)\simeq\soc(X \conv M(\bolda) \conv M'(\boldb_k \bolde_k))  \\
&&\simeq\soc(M(\boldb) \conv M'(\boldb_k \bolde_k))\simeq
M(\bar \boldb) \sconv ( M(\boldb_k \bolde_k)\sconv M'(\boldb_k \bolde_k)) \\
&&\simeq M(\bar \boldb) \sconv U^{\conv \boldb_k}\simeq
M(\bar \boldb) \conv M(\sum_{b_{ik}>0} b_{ik}  \boldb_k\bolde_i )
\simeq M'(\bar \boldb+\sum_{b_{ik}>0} b_{ik}  \boldb_k\bolde_i ).
\eneqn
Hence we have
\eqn
(\expL[{\mu_k(\shs)}](X))_i &&=(\bar \boldb+\sum_{b_{ik}>0} b_{ik}  \boldb_k\bolde_i 
 -(\bolda+\boldb_k\bolde_k))_i \\
&&= \begin{cases}
 -\boldb_k=-\boldg_k &\text{if} \quad i=k, \\
 \boldb_i-\bolda_i+ [b_{ik}]_+ \boldb_k = \boldg_i +[b_{ik}]_+\boldg_k &\text{if} \quad i \neq k. 
 \end{cases}
\eneqn

\medskip\noi
(Case $\boldb_k=0$)\hs{1.5ex} We have $\boldg_k=-\bolda_k\le0$.
 Set $\bar \bolda = \bolda-\bolda_k\bolde_k$.
Since we have
\eqn 
&&\soc\bl X \conv M'(\bar \bolda)\conv V^{\conv \bolda_k}\br \sconv M(\bolda_k\bolde_k)
=\soc(X \conv M(\bolda) \conv V^{\conv \bolda_k}) 
=M(\boldb)\sconv  V^{\conv \bolda_k}  \\
&&=M(\boldb) \sconv (M'(\bolda_k\bolde_k)\sconv M(\bolda_k\bolde_k)) 
=\bl M'(\boldb) \sconv M'(\bolda_k\bolde_k)\br\sconv M(\bolda_k\bolde_k), 
\eneqn
we obtain that 
\eqn 
\soc\bl X \conv M'(\bar \bolda) \conv V^{\conv \bolda_k}\br =M'(\boldb) \sconv M'(\bolda_k\bolde_k)
=M'(\boldb+\bolda_k\bolde_k).
\eneqn
It follows that
\eqn
(\expL[{\mu_k(\shs)}](X))_i &&=(\boldb+\bolda_k\bolde_k -\bar \bolda -\sum_{b_{ki}>0} b_{ki}  \bolda_k\bolde_i)_i \\
&&= \begin{cases}
 \bolda_k=-\boldg_k &\text{if} \quad i=k, \\
  \boldb_i- \bolda_i-[b_{ki}]_+ \bolda_k = \boldg_i +[b_{ki}]_+\boldg_k &\text{if} \quad i \neq k, 
 \end{cases}
\eneqn
as desired.
\end{proof}

 Since  the maps $\pR_{\mu_k(\shs),\shs}$ and
$\pL_{\mu_k(\shs),\shs}$ are bijections on $\Z^\K$,
Corollary \ref{cor:surj1}  implies the following result.
\begin{corollary} \label{cor:parametrization}
The map $\expR$ is a bijection from the set of classes of self-dual simples in $\Cwt$ to the set $\Z^{\K}$ for any  monoidal  seed    $\shs$.
\end{corollary}

\begin{corollary} \label{cor:parametrizationL}
The map $\expL$ is a bijection from the set of isomorphism classes of self-dual simples in $\Cwt$ to the set $\Z^{\K}$ for any   monoidal  seed    $\shs$.
\end{corollary}
\begin{proof}
It is enough to show that $\expL$ is surjective.
By  Theorem  \ref{thm:dual},  there is a map  $\mathscr D$  sending  a simple $M$ to its right dual  $\mathscr D(M)$. 
By  Lemma \ref{lem:expanddual}, for a self-dual simple $M$ in $\Cwt$, we have
$$\expL(\mathscr D(M)) =-\expR(M).$$
Since $\expR$ is surjective, so is $\expL$.
\end{proof}

Now the following is another consequence of Lemma \ref{lem:expanddual}.
\begin{corollary} \label{cor:dualformula}
Let $M$ be a self-dual simple in $\Cwt$. 
Then 
\eqn
(\expL)^{-1} (-\expR(M)) \simeq \mathscr D(M) \qtext
 \left(\expR\right)^{-1} (-\expL (M)) \simeq \mathscr D^{-1}(M) .
\eneqn
\end{corollary}

 \subsection{Injective reachable seeds and common triangular basis}
\label{subsec:common}

It is known (see, for example, \cite{Qin17}) that, for  every  monoidal cluster  $ \mcluster =\set{M_j}{j\in \K}$ of $\Cw$, 
there exists a sequence of mutations $\Sigma$ and a permutation $\sigma$ on $\Kex$ such that 
\eqn
&&\expR(M_{\Sigma(\shs)}(\bolde_{\sigma(i)})) \in -\bolde_i +\Z^{\Kfr} \quad \text{for} \ i \in \Kex, \\
&&(\Sigma b)_{\sigma(i),\sigma(j)} =b_{i,j} \quad \text{for} \ i,j \in \Kex,
\eneqn
where  $\seed$ is the monoidal seed associated with $\mcluster$,  $\set{\bolde_j \in \Z^\K}{j\in \K}$ denotes the natural basis of $\Z^\K$, 
  $\tB=(b_{i,j})_{i\in \K, j \in \Kex}$ and $\Sigma(\tB)=((\Sigma b)_{i,j})_{i\in \K, j \in \Kex}$ are the exchange matrices of the  monoidal  seeds $\shs$ and $\Sigma(\seed)$, respectively,  and  $M_{\Sigma( \mcluster )}(\bolda)$ denotes the self-dual monomials in the  monoidal cluster  $\Sigma(\mcluster)$. 
See \cite[Proposition 13.4]{GLS11} for  a  precise description of $\Sigma$. 
A  quantum seed $[\shs]$ satisfying the above condition is called \emph{injective-reachable} in \cite{Qin17}.
By Corollary \ref{cor:dualformula},
for any $k\in\Kex$, there exists 
 $\boldc_k\in\Z^{\Kfr}$ such that 
\eqn 
M_{\Sigma( \mcluster )}(\bolde_{\sigma(k)})\conv M(\boldc_k) \simeq \mathscr D^{-1}(M_k) \quad \text{up to a grading shift.}
\eneqn
Note that $M(\boldc_k)$ is a central invertible object of  $\tCw$.

 In \cite{Qin17}, Qin provided  the notion of
 \emph{common triangular bases}  for  a  quantum cluster algebra whose seeds are all injective-reachable.  He studied the existence of common triangular basis for some cases including $\Anw$.

Let $ \mcluster  =\set{M_j}{j\in \K}$ be a  monoidal cluster   in $\Cw$.
Recall that $K(\Cw)[ \mcluster ^{-1}]$ is the Laurent polynomial ring in
$\{[M_j]\}_{j\in\K}$ (see \eqref{eq:Lrpr}). 
For $\boldb\in\Z^J$, we denote by
$[M(\boldb)]$ the element in $K(\Cw)[ \mcluster ^{-1}]$
given by
$$[M(\boldb)]=q^{
\tL(M(\boldb+\bolda), M(\bolda))-\tL(M(\bolda), M(\bolda))}
[M(\boldb+\bolda)]\cdot [M(\bolda)]^{-1}$$
for any $\bolda\in \Z_{\ge0}^J$ such that $\boldb+\bolda\in \Z_{\ge0}^J$.
It does not depend on the choice of $\bolda$.
Note that we do not define
$M(\boldb)$ itself.
It coincides with the ordinary
$[M(\boldb)]$ when $\boldb\in\Z_{\ge0}^J$.

We set 
$$X^\boldb:=q^{-(\wt(M(\boldb)),\wt(M(\boldb)))/4} [M(\boldb)]  \in  \Z[q^{\pm1/2}] \otimes_{\Z[q^{\pm1}]}K(\mathcal C_w)[ \mcluster ^{-1}].$$
We call $X^\boldb$ a \emph{quantum Laurent monomial} in $[\shs]$.
 We  have 
$$X^\boldb\cdot X^{\boldb'}=q^{\frac{1}{2}\la(\boldb,\boldb')}X^{\boldb+\boldb'},$$
where $\la$ is a skew-symmetric bilinear form on $\Z^\K$ such that
$\la(\bolde_i,\bolde_j)=-\La(M_i,M_j)$.

 A $\Z[q^{\pm{1/2}}]$-basis $\mathcal B$ of $\Z[q^{\pm{1/2}}] \otimes_{\Z[q^{\pm1}]}K(\Cwt)$ is  a {\em common triangular basis}  
if it satisfies the following four conditions 
for  every  monoidal  seed  $\shs$  (see, \cite[Section 6.1]{Qin17}):
\begin{enumerate}
\item $\mathcal B$ contains the quantum cluster monomials in the quantum seed $[\shs]$. 
\item (Bar invariance) The  expansion of $b\in \mathcal B$ as a $\Z[q^{\pm\frac{1}{2}}]$-linear combination of  quantum Laurent monomials in $[\shs]$ has bar-invariant coefficients.
\item (Parameterization) The  expansion of $b\in \mathcal B$ as a $\Z[q^{\pm\frac{1}{2}}]$-linear combination of the quantum Laurent monomials in $[\shs]$ contains a unique non-zero term $c_\bolda X^{\bolda}$ such that 
all the other non-zero 
terms $c_\boldb X^\boldb$  satisfy that  $\bolda \succ_\shs \boldb$. 
Moreover, the map  $\mathcal B \to \Z^J$, given by $b \mapsto \bolda$, is a bijection.

\item (Triangularity) For any quantum cluster variable $X_i$ in the seed $[\shs]$ and for any element $S$ in $\mathcal B$, 
 we can write
\eq \label{eq:triangularity} 
X_i   \, S = \sum_{k=1}^t \eps_kq^{\ell_k}b_k
\eneq
for some $\eps_k=\pm1$, $b_k\in \mathcal B$ and $\ell_k\in\Z/2$ ($1\le k\le t$)
such that
$\eps_1=1$,
$\expR(b_1)=\expR(X_i)+\expR(S)$,
$\expR(b_k) \prec_{\shs} \expR(b_1)$ 
and $\ell_1>\ell_k$ if $2\le k\le t$. 
\end{enumerate}

The purpose of this subsection is to prove the following proposition.
Recall that $\Irrsd(\tCw)$ is the set of the isomorphism classes of self-dual simple objects of $\tCw$. 
\Prop\label{prop:triangular} Assume that 
the base field $\cor$ of the quiver Hecke algebra is of characteristic $0$.
Then the basis $\mathcal B:=\set{ q^{-(\wt(S),\wt(S))/4}[S]}{S \in \Irrsd(\Cwt)}$ is a common triangular basis of
the quantum cluster algebra $\Z[q^{\pm1/2}]\otimes_{\Z[q^{\pm1}]}K(\Cwt)$.
\enprop

\begin{remark}
 Proposition~\ref{prop:triangular} is not new.
Combining \cite[Theorem 6.1.6]{Qin17}
and Theorem \ref{thm:moncat}, one can conclude that  the basis  
$\mathcal B$ is a common triangular basis, 
as pointed out at the end of \cite[Section 1.2]{Qin17}. 

But we will give  here a  more direct verification of the conditions  (2)--(4)  rather than using \cite[Theorem 6.1.6]{Qin17}. 
\end{remark}

\bigskip
\Proof[{Proof of Proposition~\ref{prop:triangular}}]
\hs{1ex}Condition (1) follows from Theorem \ref{thm:moncat}.
We show the following lemma for  the proof of condition (2).
\begin{lemma} \label{lem:bar}
Let $ \mcluster  =\set{M_j}{j\in \K}$ be a  monoidal cluster   in $\mathcal C_w$ and   $S$ a self-dual simple module in $\mathcal C_w$. Assume that 
\eq \label{eq:expan}
[S \conv  M(\bolda)] = \sum_{k=1}^r p_k(q) [M(\boldb(k))]
\eneq
for some $\bolda, \boldb(1), \ldots, \boldb(r) \in \Z_{\ge 0}^{\K}$ such that $\boldb(k)\neq \boldb(k')$ for $k\neq k'$, and $p_k(q)\in \Z[q^{\pm1}]$.
Then 
 \eqn 
p_k(q) q^{\frac{1}{2}(\wt(M(\bolda)),\wt(S))+\frac{1}{2} \La(M(\boldb(k)),M(\bolda))} \quad
\text{\emph{is bar-invariant for all}} \ 1\le k \le r.
\eneqn
\end{lemma}
\begin{proof}
Multiplying $[M(\bolda)]$ to \eqref{eq:expan} from the left, we have
\eqn
 [M(\bolda)\conv S \conv M(\bolda)] = \sum_{k=1}^r p_k(q) q^{-\tLa(M(\bolda),M(\boldb(k)))} [M(\boldb(k)+\bolda)].
\eneqn

On the other hand, taking $-$ to \eqref{eq:expan} and multiplying $[M(\bolda)]$ from the right, we have
\eqn
q^{(\wt (M(\bolda), \wt(S))} [M(\bolda)\conv S \conv M(\bolda)] = \sum_{k=1}^r p_k(q^{-1})  q^{-\tLa(M(\boldb(k)),M(\bolda))} [M(\boldb(k)+\bolda)].
\eneqn
Comparing the coefficients of $[M(\boldb(k)+\bolda)]$, we get the desired result.
\end{proof}
\newcommand{\mua}{\mu_{\bolda}}
\newcommand{\muc}{\mu_{\boldc}}
\newcommand{\mubk}{\mu_{\boldb(k)}}

 Let us prove condition (2). 
For  a self-dual simple module  $S$  in $\mathcal C_w$, write
\eqn
[S \conv  M(\bolda)] = \sum_{k=1}^r p_k(q) [M(\boldb(k))],
\eneqn
where   $\bolda, \boldb(1), \ldots, \boldb(r) \in \Z_{\ge 0}^{\K}$ such that $\boldb(k)\neq \boldb(k')$ for $k\neq k'$, and $p_k(q)\in \Z[q^{\pm1}]$.
Then we have
\eqn
q^{\frac{1}{4}(\mu_\bolda,\mu_{\bolda}) }[S] X^\bolda = \sum_{k=1}^r p_k(q)q^{\frac{1}{4}((\mu_{\boldb(k)},\mu_{\boldb(k)})} X^{\boldb(k)},
\eneqn
where $\muc=\wt(M(\boldc))$ for $\boldc \in \Z_{\ge 0}^\K$.
It follows that 
\eqn
q^{\frac{1}{4}(\mu_\bolda,\mu_{\bolda})} [S]
&&= \sum_{k=1}^r p_k(q) q^{\frac{1}{4}(\mubk,\mubk)} X^{\boldb(k)} X^{-\bolda} \\*
&&= \sum_{k=1}^r p_k(q)  q^{\frac{1}{4}(\mubk,\mubk)+\frac{1}{2}\La(M(\boldb(k)),M(\bolda))} X^{\boldb(k)-\bolda}
\eneqn
and hence
\eqn
q^{-\frac{1}{4}(\wt(S),\wt(S))} [S]
&&= 
q^{\frac{1}{4}(\mu_\bolda,\mu_{\bolda})+\frac{1}{2}(\mua,\wt(S))-\frac{1}{4} (\mubk,\mubk)} [S] \\*
&&= \sum_{k=1}^r p_k(q)  q^{\frac{1}{2}\La(M(\boldb(k)),M(\bolda))+\frac{1}{2}(\mua,\wt(S) )} X^{\boldb(k)-\bolda}.
\eneqn
Thus   condition (2) follows by
 Lemma  \ref{lem:bar}. 
 
 The map in  condition (3) is $\expR$ 
 by Lemma~\ref{lem:Mab},  and hence 
 the last assertion in condition (3) is nothing but 
Corollary \ref{cor:parametrization}. 
\smallskip

 It remains to prove  condition (4). 
By taking $-$ to \eqref{eq:triangularity}, it is equivalent to showing that
for any self-dual simple module $S$ in $\mathcal C_w$ and a cluster variable $M_i$ in the seed $\shs$, we have
\eqn 
[S \conv M_i] = \sum_{j=1}^t\eps_jq^{\ell_j}[S_j]
\eneqn
where $S_j$ is a self-dual simple module in $\mathcal C_w$ 
and $\eps_j=\pm1$, $\ell_j\in\Z/2$ ($1\le j\le t$)
such that 
$\expR(S_1)=\expR(M_i)+\expR(S)$, 
and $\ell_1<\ell_j$ and
$\expR(S_k) \prec_{\shs} \expR(S_1)$ for $2\le k\le t$.

Since $M_i$ is a real simple module, there exist self-dual simple modules $S_j$ and $\ell_j\in\Z$ such that 
\eqn
[ S\conv M_i]=\sum_{j=1}^t q^{\ell_j} [S_j] ,
\eneqn
$S\hconv M_i\simeq q^{\ell_1}S_1$ 
and $\ell_1 < \ell_j$, $\La(S_j, M_i)<\La(S, M_i)$  for $2\le j\le t$
(\cite[Theorem 4.2.1, Theorem 4.1.1]{KKKO18} ).

Hence it remains to prove
$\expR(S_j) \prec_{\shs} \expR(S_1)$ for $2\le j\le t$.

Assume that 
\eqn
[S \conv  M(\bolda)] = \sum_{k=1}^r q^{c_k} [M(\boldb(k))]
\eneqn
for some $\bolda, \boldb(k)\in \Z_{\ge0}^\K$ with $\expR(S)=\boldb(1)-\bolda$ (or $S\hconv M(\bolda)\simeq M(\boldb(1))$).
 Hence we have $\boldb(k) \prec_{\shs}\boldb(1)$ for $2\le k\le r$.

Now we have
\eq
\sum_{j=1}^t q^{\ell_j} [S_j \conv M(\bolda)] &&= [S\conv M_i\conv M(\bolda)] \nn\\
&&= q^{-\La(M_i,M(\bolda))} [ S \conv M(\bolda) \conv M_i ]\nn\\
&&=q^{-\La(M_i,M(\bolda))}  \sum_{k=1}^r q^{c_k-\tLa( M(\boldb(k)),M_i)} [M(\boldb(k)+\bolde_i)].\label{eq:SM}
\eneq
Thus for each $1\le j\le t$, $S_j\hconv M(\bolda)$ is isomorphic to 
$M(\boldb(k_j)+\bolde_i)$ for some $1\le k_j \le r$ up to a grading shift.

Note that  $S\conv M_i\conv M(\bolda)$ 
has a simple head $q^{\ell_1}S_1 \hconv M(\bolda)$ and hence  we have 
 $$S_1 \hconv M(\bolda) \simeq M(\boldb(1)+\bolde_i)
\qt{up to a grading shift.} $$ 
Hence $k_1=1$.
Since $M(\boldb(1)+\bolde_i)$ appears once in \eqref{eq:SM} 
by Lemma~\ref{lem:Mab}, 
we have
$k_j\not=1$ for $j\not=1$.

It follows that 
$$\expR (S_1) = (\boldb(1)+\bolde_i-\bolda)
 \succ_{\shs} \ ( \boldb(k_j)+\bolde_i -\bolda)  =\expR (S_j) \quad \text{for } \ 2\le j \le t.  $$
Thus we obtain condition (4), and the proof of  
Proposition~\ref{prop:triangular} is complete.
\QED

\begin{remark}
Qin showed  that  the upper global basis $\mathcal B$ is a common triangular basis  of $\Z[q^{\pm1/2}] \otimes_{\Z[q^{\pm1}]}K(\Cwt)$ if 
$w$ has an \emph{adaptable reduced expression} (for the definition, see \cite[Section 8.2]{Qin17}).
\end{remark}
\hfill

\section{Commutativity with cluster variables and  denominator vectors} \label{sec:commden}

In this section, we study the denominator vectors introduced in \cite{FZ02}
in the context of monoidal categorification.
As a byproduct we give a proof to some conjectures of
Fomin-Zelevinsky  (\cite{FZ07})  in the case of $\Anw|_{q=1}$. 

\begin{theorem} \label{thm:commcv}
Let $X$ be a simple module and $M$  a real simple module in $R\gmod$.
If $[X] = [M] \phi$   for some $\phi$ in ${K(R\gmod)|}_{q=1}$,
then $X \simeq M\conv Y$ for some simple $R$-module $Y$ which strongly commutes with $M$.
\end{theorem}
\begin{proof}
We may assume that
\eqn
\phi=\sum_{i \in K }[Y_i]-\sum_{j\in K'}[Z_j],
\eneqn
where $Y_i$ and $Z_j$ are simple $R$-modules and there is no pair $(i,j)\in K\times K'$  such that $Y_i\simeq Z_j$.
It follows that 
\eqn 
[X] +\sum_{j\in K' }[M\conv Z_j]= \sum_{i \in K}[M\conv Y_i].
\eneqn
Take $i_0$ such that $\La(M,Y_{i_0})=\max\set{\La(M,Y_i)}{i\in K}$. 
 For $j \in K'$, the head $M \hconv Z_j$ appears as a subquotient of  some $M \conv Y_i$.
Since $M \hconv Z_j \not \simeq  M \hconv Y_i$, we have
\eqn 
\La(M,Z_j)=\La(M,M\hconv Z_j) < \La(M,M \hconv Y_i) = \La(M,Y_i) \le  \La(M,Y_{i_0}).
\eneqn
 Since any simple  subquotient $S$ of $M\conv Z_j$ satisfies 
\eqn \La(M,S) \le \La(M,Z_j) <\La(M,Y_{i_0}) =\La(M,M\hconv Y_{i_0}), \eneqn
we conclude that
 $M\hconv Y_{i_0}$ does not appear in $M\conv Z_j$ for any $j \in K'$.
 Hence $$X\simeq M\hconv Y_{i_0}.$$

In particular, we have
$$\La(M, Y_i)\le \La(M,Y_{i_0})=\La(M,M\hconv Y_{i_0})=\La(M, X)
\qt{for any $i\in K$.}$$

By a similar reasoning, we have 
\eqn 
&&\La(Y_i,M) \le \La(X,M) \quad \text{for any $i$.}
\eneqn

In particular, we have
$$\La(Y_{i_0},M)\le \La(X,M) = \La(M\hconv Y_{i_0},M) .$$
Then by \cite[Corollary 4.1.2 (i)]{KKKO18}, we conclude that $M$ and $Y_{i_0}$ strongly commute
and $X \simeq M \conv Y_{i_0}$, as desired.
\end{proof}

\begin{lemma} \label{lem:commcv1}
Let $\set{M_j}{j\in \K}$ be a   monoidal cluster   in $\mathcal C_w$. 
Let $X$ be a simple module in $\mathcal C_w$ and $k \in \Kex$.
Assume that $$[X \conv M(\bolda)] = \sum_s [M(\boldb(s))] \quad 
 \text{in} \ K(\mathcal C_w)|_{q=1} 
 \quad \text{for some } \ \bolda, \boldb(s) \in \Z_{\ge 0}^\K.$$
If $\bolda_k=0$, then
$X$ strongly commutes with $M_k$.
\end{lemma}
\begin{proof}
Write 
$[X\conv M_k] =\displaystyle \sum_i[S_i]$, where $S_i$'s are some  simple modules.
Then 
\eqn 
\sum_s [M(\boldb(s))\conv M_k]  = [X \conv M(\bolda) \conv M_k] = \sum_i [S_i \conv M(\bolda)].
\eneqn
Hence for every $i$, the element
$[S_i ]\cdot[M(\bolda)]$ is a sum of elements of
 the form $[M(\boldb(s))]\cdot[M_k]$.
It follows that $[M_k]$ divides $[S_i]$, because  $K(R\gmod)|_{q=1} \simeq \An |_{q=1}$ is a factorial ring 
 and the cluster variables $[M_j]$ are  prime elements in it (\cite{GLS13}). 
 In other word, we have $[S_i] = \phi_i [M_k]$ for some $\phi_i$ and hence by Theorem \ref{thm:commcv}, there exists a simple module $K_i$  such that it strongly commutes with $M_k$ and $S_i \simeq K_i \conv M_k$  for all $i$.

It follows that $$[X][M_k]=\displaystyle \sum_i[S_i]= \sum_i [K_i][M_k]$$
and hence $[X]=\displaystyle \sum_i [K_i]$. Since $X$ is simple, we conclude that
$X\simeq K_i$ for some $i$. 
\end{proof}

\begin{lemma}\label{lem:commcv2}
Let $  \mcluster  = \set{M_j}{j\in \K}$ be a  monoidal cluster   in $\mathcal C_w$. 
For every simple $X$ in $\mathcal C_w$ and $k \in \Kex$, there exist $\bolda, \boldb(s), \boldc(t) \in \Z_{\ge 0}^\K$ such that $\bolda_k=0$, $\boldc(t)_k>0$, 
and
\eqn
[X \conv M(\bolda)]=\sum_s [M(\boldb(s))] +\sum_t [M'(\boldc(t))] \quad \text{in} \ K(\mathcal C_w)|_{q=1},
\eneqn
where $M'(\boldc)$ for $\boldc \in \Z^\K$ denotes a cluster monomial of  the  monoidal cluster    $\mu_k( \mcluster  )$, the mutation of  $ \mcluster  $ at $k$.
\end{lemma}

\begin{proof}
We can write
\eqn
[X \conv M(\bolda)] = \sum_s [M(\boldb(s))] \quad \text{with} \ \bolda \in \Z_{\ge 0}^\K, \boldb(s)  \in \Z^\K.
\eneqn

Write $\bolda=\overline \bolda + \bolda_k \bolde_k$, where $\set{\bolde_j}{j \in \K}$ is  the natural basis of $\Z^J$. 

Set $V := M\left (\sum_{b_{ki}> 0} b_{ki}\bolde_i\right)$, 
and $U := M\left(\sum_{b_{ik} >0} b_{ik}\bolde_i\right)$ so that
$$[M_k \conv M_k'] =[U] + [V].$$

Set $\boldb(s)=\overline{\boldb(s)} + (\boldb(s))_k \bolde_k$.
By multiplying $[{M'_k}^{\circ \bolda_k}]$, we obtain
\eqn 
[X \conv M(\overline \bolda)]([U]+[V])]^{\bolda_k} =&&
\sum_{\boldb(s)_k \ge \bolda_k} [M(\boldb(s)-\bolda_k \bolde_k)]([U]+[V])]^{\bolda_k} \\*
&&\hs{5ex}+ \sum_{\boldb(s)_k < \bolda_k} [M(\overline{\boldb(s)})]([U]+[V])]^{\boldb(s)_k}[M'_k]^{\bolda_k-\boldb(s)_k}.
\eneqn
The right hand side is a sum of $[M(\boldb)]$'s and $[M'(\boldc)]$'s such that $\boldc_{k} >0$ ($\boldb,\boldc \in \Z_{\ge 0}^{\K}$),
 while  $[X \conv M(\overline{\bolda}) \conv U^{\circ \bolda_k}]$ is a component of the left hand side. 
Hence, by the following sublemma, 
$[X \conv M(\overline{\bolda}) \conv U^{\circ \bolda_k}]$ is a sum of $[M(\boldb)]$'s and $[M'(\boldc)]$'s, as desired.
\end{proof}
\begin{sublemma}
Let $\{X_a\}_{a\in A }$ and $\{Y_b\}_{b\in B}$ 
be finite families of simple modules in $R\gmod$.
If \eqn \sum_{a\in A}[X_a]=\sum_{b\in B}[Y_b] \quad \text{in} \ K(R\gmod)|_{q=1},  \eneqn
then
there exists a bijection $f:A\simeq B$ such that $X_a \simeq Y_{f(a)}$ up to a grading shift.
\end{sublemma}
\begin{proof}
It is immediate from the fact that the isomorphism classes of self-dual simple modules in $R\gmod$ forms a $\Z[q^{\pm 1}]$-basis of $K(R\gmod)$.
\end{proof}

\begin{lemma}\label{lem:commcv3}
Let $\set{M_j}{j\in \K}$ be a   monoidal cluster   in $\mathcal C_w$. 
Let $X$ be a simple module which strongly commutes with $M_k$ for some $k \in \K$.
Then we have  
\eqn
[X\conv M(\bolda)] = \sum_{s} [M(\boldb(s))] \quad \text{in} \ K(\mathcal C_w)|_{q=1} 
\eneqn
for some $\bolda, \boldb(s) \in \Z_{\ge0}^\K$ with $\bolda_k=0$.
\end{lemma}
\begin{proof}
Write 
$$[X \conv M(\bolda)]=\sum_s [M(\boldb(s))] +\sum_{t\in K} [M'(\boldc(t))],$$
with $\bolda, \boldb(s), \boldc(t) \in \Z_{\ge 0}^\K$ such that  $\bolda_k=0$ and $\boldc(t)_k> 0$ for $t \in K$.
Since $X$ and $M(\bolda)$ strongly commute with $M_k$, every subquotient of $X \conv M(\bolda)$ strongly commutes with $M_k$ by \cite[Proposition 3.2.10]{KKKO18}. 
In particular, if $t\in K$, then $M'(\boldc(t))$ strongly commutes with $M_k$ so that 
$$0=\de(M_k,M'(\boldc(t))) = \de(M_k,{M'_k}^{\circ \boldc(t)_k}) = \boldc(t)_k >0, $$
which is a contradiction. Hence $K=\emptyset$ and the assertion follows.
\end{proof}

Let $X$ be a simple module in $\mathcal C_w$ and $ \mcluster  =\set{M_j}{j\in\K}$ a   monoidal cluster   in $\mathcal C_w$. 
Write $[X]=\sum_s [M(\boldb(s))]$ with $\boldb(s) \in \Z^\K$ in 
$K(\mathcal C_w)[ \mcluster ^{-1}]|_{q=1}$. 
Then the vector $\boldd(X) \in \Z^{\K}$ defined by
$$\boldd(X)_k \seteq -\min_{s}\{(\boldb(s))_k\} \quad \text{for} \ k \in \Kex, \qtext  \boldd(X)_j=0 \quad \text{for } \ j \in \Kfr$$
is called the \emph{denominator vector of $[X]$} in the cluster algebra $K(\mathcal C_w)|_{q=1}$, which is introduced in \cite{FZ02}. 

By Lemma \ref{lem:commcv1} and Lemma \ref{lem:commcv3}, we obtain the following proposition.

\begin{prop} \label{prop:commcv}
Let  $ \mcluster  =\set{M_j}{j\in \K}$ be a   monoidal cluster    in $\mathcal C_w$.
If $X$ is a simple in $\mathcal C_w$, $k \in \Kex$, and $[X] = \sum_{s=1}^{N}[M(\boldb(s))]$ in $K(\mathcal C_w)[ \mcluster ^{-1}]|_{q=1}$ for some $\boldb(s) \in \Z^J$,
then the followings are equivalent:

\bna
\item
$X$ strongly commutes with $M_k$,
\item 
$(\boldb(s))_k \ge 0$ for all $s$,
\item
 $\boldd(X)_k\le0$.
\end{enumerate}
\end{prop}

\Prop\label{prop:Mdivisible}
Let  $ \mcluster  =\set{M_j}{j\in \K}$ be a   monoidal cluster    in $\mathcal C_w$.
If $X$ is a simple in $\mathcal C_w$, $k \in \Kex$, and 
assume that $X$ strongly commutes with $M_k$.
Then, $-\boldd(X)_k$ is the largest non-negative integer $n$ such that
there exists a simple $Y$ such that $X \simeq Y\conv M_k^{\circ n}$.
\enprop
\begin{proof}
By the assumption, $\boldd(X)_k\le0$.
Assume that there exists a simple $Y$ such that $X \simeq Y \conv M_k^{\circ n}$.
We shall show $  -\boldd (  X)_k\ge n$.
We may assume that
$n>0$. Then $Y$ strongly commutes with $M_k$. 
By the above proposition, 
$[Y] = \sum_{s=1}^{N}[M(\boldb'_s)]$ for some $\boldb'_s \in \Z^J$ with $(\boldb'_s)_k \ge 0$ for all $s$.
Hence we have $[X] =[Y\conv M_k^{\circ n}]= \sum_{s=1}^{N}[M(\boldb'_s+n\bolde_k)]$ in $K(\mathcal C_w)[ \mcluster ^{-1}]|_{q=1}$, as desired.

Conversely, assume that $-\boldd(X)_k\ge n \ge0$.
Write
$[X\conv M(\bolda)] = \sum_{s\in \Sigma}[M(\boldb(s))]$ in $K(\mathcal C_w)[ \mcluster ^{-1}]|_{q=1}$. Since
$-\boldd(X)_k=\min\set {(\boldb(s)-\bolda)_k}{s\in\Sigma}\ge n$, we may assume that
$\bolda_k=0$ and $\boldb(s)_k\ge n$ for every $s\in \Sigma$.
Hence $[X]\cdot[M(\bolda)]=[X\conv M(\bolda)]$
is divisible by $[M_k]^n$ in $K(\Cw)$.
Note that any cluster variable is a prime element of 
$K(\Cw)|_{q=1}$ by \cite{GLS13}.
Since $[M(\bolda)]$ is not divisible by $[M_k]$ and
$[M_k]$ is prime, $[X]$ is divisible by $[M_k]^n$.
 Then, by Theorem \ref{thm:commcv},  we have 
$X \simeq  M_k^{\circ n}\conv Y$ for some simple module $Y$.
\end{proof}

\begin{corollary} \label{cor:denominator}
Let $\set{M_j}{j\in \K}$ be a   monoidal cluster   and 
let $[X]$ be  a cluster variable  which does not belong to $\set{M_j}{j\in \K}$. Then
\bnum
\item $\boldd(X)_k \ge 0$ for all $ k \in \Kex$.
\item If there is a seed containing  $X$ and $M_k$ for some $k \in \Kex$,  then $\boldd(X)_k = 0$.
\end{enumerate}
\end{corollary}
\begin{proof}
(i) 
If $X$ doesn't commute with $M_k$, then  $\boldd(X)_k > 0$.
Assume that $X$ commutes with $M_k$. Then we have $\boldd(X)_k \le 0$. 
Since $[X]$ is a cluster variable and hence is a prime element, it cannot be factored into $X \simeq Y \conv M_k$. Thus we have $\boldd(X)_k=0$
 by Proposition~\ref{prop:Mdivisible}. 

\medskip\noi
(ii) Since $X$ and $M_k$ strongly commute 
with each other, we have $\boldd(X)_k\le0$. 
\end{proof}

\begin{remark}
Corollary \ref{cor:denominator} (i) is the Conjecture 7.4 (1) in \cite{FZ07}. 
Corollary \ref{cor:denominator} (ii) is a half of the Conjecture 7.4 (2) in \cite{FZ07}. 
\end{remark}

\begin{theorem}[see also Theorem~\ref{th:clusterch}]\label{th:clusterch1}
Let $X$ be a simple module in $\mathcal C_w$ and let  $ \mcluster    =\set{M_j}{j\in \K} $ be a  monoidal cluster   in $\mathcal C_w$. 
 Let $k \in \Kex$ and $\shs'=\mu_k(\shs)$. Assume that  $X$ strongly commutes with $M_j$ for all $ j \in \K\setminus \{k\}$.
Then either $X$ strongly commutes with $M_k$ or it strongly commutes with $M_k'$, the mutation of $M_k$  at $k$. 
 Equivalently, $[X]$ is a cluster monomial in the seed $[\shs]$ or a cluster monomial in the seed $[\shs']$ \ro up to a power of $q$\rf.

\end{theorem}
\begin{proof}
 For $\bolda\in\Z_{\ge0}^\K$, we will denote by
$M'(\bolda)$ the cluster monomial in $\mu_k(\shs)$. 
By the assumption we have $\boldd(X)_j \le 0$ for all $ j \in \K\setminus \{k\}$. 
Assume that  $X$ does not strongly commute with $M_k$.  Then  we have $\boldd(X)_k >0$.
It follows that 
$$[X \conv M_k^{\circ \boldd(X)_k }] =\sum_s [M(\boldb(s))]
\qt{in $K(\Cw)\vert_{q=1}$}$$
for some $\boldb(s) \in \Z_{\ge 0}^{\K}$.
Multiplying ${M_k'}^{\boldd(X)_k }$, we obtain
$$[X]\cdot\bl [U]+[V]\br^{ \boldd(X)_k } =\sum_s [M(\boldb(s))]\cdot[M'(\boldd(X)_k\bolde_k)],$$
where $[M_k]\cdot[M_k']=[U]+[V]$. 
Note that $U$ and $V$ are cluster monomials in the  monoidal cluster   $\mu_k( \mcluster  )$ and they strongly commute with $X$. Hence the left-hand side of the above equation is a sum of elements of the form
$[X \conv M'(\bolda)]$ for some $\bolda \in \Z_{\ge 0}^\K$ with $\bolda_k=0$, 
which is a class of simple module.

On the other hand, by the definition of $\boldd(X)_k$, there exists $s_0$ such that $\boldb(s_0)_k=0$ and hence $M(\boldb(s_0))=M'(\boldb(s_0))$.
 Thus there exists $\bolda\in \Z_{\ge 0}^{\K}$ such that  $\bolda_k=0$ and 
   $$[X \conv M'(\bolda)]=[M'(\boldb(s_0)+\boldd(X)_k\bolde_k)].$$ 
Thus we obtain 
$$[X]=[M'(\boldb(s_0)+\boldd(X)_k\bolde_k -\bold a)]$$
in $K(\mathcal C_w)[\mu_k( \mcluster  )^{-1}]\vert_{q=1}$. 
Since $X \in \mathcal C_w$, we conclude that $\boldb(s_0)+\boldd(X)_k\bolde_k -\bold a \in \Z_{\ge 0}^\K$.
 Hence $[X]$ is a cluster monomial in  the cluster  $[\mu_k( \mcluster  )]$. 
\end{proof}

\medskip
\begin{prop} \label{prop:commu}
Let $\set{M_j}{j\in \K}$ be a   monoidal cluster   in $\mathcal C_w$, $k \in \Kex$, and  $X$ a simple module in $\mathcal C_w$. 
Assume that $X$ 
does not strongly commute with  $M_k$.
Then $\boldd(X)_k$ is the smallest positive 
integer $n$ such that 
every simple subquotient of $X\conv M_k^{\circ n}$ strongly commutes with $M_k$. 
\end{prop}
\begin{proof}
Note that $\boldd(X)_k >0$, since $X$ 
does not strongly commute with  $M_k$.

Write $[X\conv M(\bolda)] = \sum_{s \in \Sigma} [M(\boldb(s))]$ in $K(\Cw)\vert_{q=1}$
with $\bolda, \boldb(s) \in \Z_{\ge0}^J$.

Assume that $n\in\Z_{>0}$ and
every subquotient of $X\conv M_k^{\circ n}$ strongly commutes with $M_k$.
Write $[X \conv  M_k^{\circ  n }]=\sum_t [S_t]$ 
in $K( \mathcal C_w )\vert_{q=1}$ where $S_t$ are simple modules in $\mathcal C_w$.
Then 
$$\sum_{s \in \Sigma}[M(\boldb(s) + n\bolde_k)] 
=[X\conv M(\bolda) \conv M_k^{\circ n}]=\sum_t [S_t\conv M(\bolda)] \qt{in $K(\Cw))\vert_{q=1}$.}$$
Hence  for each $t$, there exists a subset $\Sigma_t \subset \Sigma$ such that 
$$[S_t\conv M(\bolda)] = \sum_{s \in \Sigma_{t}} [M(\boldb(s) +n\bolde_k)].$$ 
We may assume that $\Sigma=\bigsqcup_{\;t}\Sigma_t$.

Thus $S_t$ strongly commutes with $M_k$ if and only if 
$\boldb(s)_k +n-\bolda_k\ge  0$ for all $s\in \Sigma_t$ by Proposition \ref{prop:commcv}.

Hence $S_t$  strongly commutes with $M_k$ for any $t$
if and only if $\boldb(s)_k +n-\bolda_k\ge  0$ for all $s\in \Sigma$, i.e., 
$\boldd(X)_k=\max\set{\bolda_k-\boldb(s)_k}{s\in\Sigma}\ge n$.
\end{proof}

This characterization of $\boldd(X)_k$ (Propositions~\ref{prop:commcv},
\ref{prop:Mdivisible} and \ref{prop:commu}) immediately implies
\begin{corollary}
Let $[M_k]$ be  a cluster variable and  let $X$ be a simple module which is not divisible by $M_k$.
Then $\boldd(X)_k$ does not depend on the choice of a cluster containing $M_k$.
\end{corollary}
\begin{remark} It is Conjecture 7.4 (3) in \cite{FZ07}.
Note that \cite[Conjecture 7.4]{FZ07} is proved for arbitrary skew-symmetrizable cluster algebras in \cite{CL182} in a different way.
\end{remark}


\begin{thebibliography}{99}
%
%\bibitem{Ariki} S. Ariki, {\em On the decomposition numbers of the
%Hecke algebra of $G(M,1,n)$}, J. Math. Kyoto
%Univ. {\bf36} (1996), 789-808.


%\bibitem{BBDG}
%A. A. Beilinson, J. Bernstein and P. Deligne, \emph{Faisceaux pervers},
%Ast\'erisque {\bf100}, Soc. Math. France, Paris, 1982.


%\bibitem{BZ93} A. Berenstein and A. Zelevinsky, {\em String bases for quantum groups of type $A_r$},I. M. Gel'fand Seminar,  51--89, Adv. Soviet Math. {\bf16},Part 1, Amer. Math. Soc., Providence, RI, 1993.



\bibitem{BZ05}
%\bysame,
 A. Berenstein and A. Zelevinsky, 
 {\em  Quantum cluster algebras}, Adv. Math. {\bf195} (2005), no. 2, 405--455.


%\ber\bibitem{CL18} 
%P. Cao and F. Li, {\em Positivity of denominator vectors of skew-symmetric cluster algebras},  J.  Algebra, {\bf515} (2018), 448-–455. \er


\bibitem{CL182}
P. Cao and F. Li, 
\emph{The enough $ g $-pairs property and denominator vectors of cluster algebras}, arXiv:1803.05281v2.


 \bibitem{Casbi} E. Casbi, \emph{Dominance order and monoidal categorification of cluster algebras},   arXiv:1810.00970. 
 
 

%\bibitem{CKLP12}G. Cerulli Irelli, B. Keller, D. Labardini-Fragoso and P. Plamondon, {\em Linear independence of cluster monomials forskew-symmetric cluster algebras},
% Compos. Math.  {\bf 149}  (2013),  no. 10, 1753--1764.

%\bibitem{DMSS}B. Davison, D. Maulik, J. Schuermann, and B. Szendroi, \emph{Purity for graded potentials and quantum cluster positivity}, Compos. Math. {\bf151} (2015), no.10, 1913--1944.



 \bibitem{FG09} V. V. Fock and A. B. Goncharov, 
 \emph{Cluster ensembles, quantization and the dilogarithm}, Ann. Sci. 
\'Ec. Norm. Sup\'er., {\bf42} (2009), no. 4,  865--930. 


%\bibitem{FZ99} S. Fomin and A. Zelevinsky,  \emph{Double Bruhat cells and total positivity}, J. Amer. Math. Soc. \textbf{12}(2) (1999),335--380.


\bibitem{FZ02}
 S. Fomin and A. Zelevinsky,  \emph{Cluster algebras I. Foundations},
J. Amer. Math. Soc. {\bf 15} (2002), no. 2, 497--529. %(electronic). MR1887642 (2003f:16050).

\bibitem{FZ07}
\bysame,  \emph{Cluster algebras IV. Coefficients},
Compos. Math., {\bf143} (2007), no. 1,  112--164.



%\bibitem{GLS05} C. Gei\ss, B. Leclerc and J. Schr\"oer, \emph{Semicanonical bases and preprojective algebras}, Ann. Sci. \'Ecole Norm. Sup. (4) {\bf38} (2005), no. 2, 193--253.

\bibitem{GLS11} C. Gei\ss, B. Leclerc and J. Schr\"oer, \emph{Kac-Moody groups and
cluster algebras}, Adv. Math. {\bf 228} (2011), 329--433.

%\bibitem{GLS07}\bysame , \emph{Cluster algebra structures and semicanonical bases for unipotent groups}, arXiv:0703039\/v4 [math.RT].

\bibitem{GLS13} \bysame,
{\em Factorial cluster algebras}, Doc. Math. {\bf 18}  (2013), 249--274.

\bibitem{GLS}
\bysame,
{\em Cluster structures on quantum coordinate rings}. Selecta Math. (N.S.) {\bf19} (2013), no. 2, 337--397.

\bibitem{HL10} D. Hernandez and B. Leclerc, {\em Cluster algebras and quantum affine algebras}, Duke Math. J. {\bf 154} (2) (2010), 265--341.

%\bibitem{HL13} \bysame, \emph{Monoidal categorifications of cluster algebras of type $A$ and $D$}, in Symmetries, Integrable Systems and Representations, Springer Proc. Math. Stat., {\bf 40} (2013), 175--193.

%\bibitem{HL11}
%\bysame,
%\emph{Quantum Grothendieck rings and derived Hall algebras},
%to appear in J. reine angew. Math., DOI: 10.1515/crelle-2013-0020.
%arXiv:1109.0862v2.
%
%\bibitem{Kac}
%V. Kac,
%{\em Infinite Dimensional Lie Algebras},
%3rd ed., Cambridge University Press, Cambridge, 1990.


%
\bibitem{KKK18}
S.-J. Kang, M. Kashiwara and M. Kim, {\em Symmetric quiver
Hecke algebras and R-matrices of quantum affine algebras},
  Invent. Math., {\bf211} no.2 (2018), 591--685. 
\
%
\bibitem{KKKO15}
S.-J. Kang, M. Kashiwara,  M. Kim  and   S.-j. Oh,
\newblock{\em Simplicity of heads and socles of tensor products},
Compos. Math. {\bf151}  (2015), no.2, 377--396.


\bibitem{KKKO18}
S.-J. Kang, M. Kashiwara,  M. Kim  and   S.-j. Oh,
 \newblock{\emph{Monoidal categorification of cluster algebras}}. J. AMS, {\bf31} (2018), 349--426.

%
%%
%%%%
%\bibitem{KKK2}
%\bysame, {\em Symmetric quiver Hecke algebras and R-matrices of
%quantum affine algebras II}, arXiv:1308.0651\,v1.

%
%\bibitem{Kas02}
%M. Kashiwara, {\em On level zero representations of quantum
%affine algebras}, Duke. Math. J. {\bf112} (2002), 117--175.
%

%\bibitem{Kash91} M. Kashiwara, {\em On crystal bases of the $q$-analogue of universal enveloping algebras}, Duke Math. J. {\bf 63} (1991), 465--516.

%\bibitem{Kas93} \bysame, \emph{Global bases of quantum groups}, Duke Math. J. \textbf{69} (1993), 455--485.

%\bibitem{Kas93a} \bysame,{\em Crystal base and Littelmann's refined Demazure character formula}, Duke Math. J. {\bf 71} (1993), 839--858.

%\bibitem{Kash94} \bysame, \emph{Crystal bases of modified quantized enveloping algebra}, Duke Math. J. \textbf{73}(1994), no. 2, 383--413.

%\bibitem{Kas95} \bysame, {\em On crystal bases},  Representations of groups (Banff, AB, 1994), 155--197, CMS Conf. Proc. {\bf16}, Amer. Math. Soc., Providence, RI, 1995.

\bibitem{KKOP18} M. Kashiwara, M. Kim, S-j. Oh and E. Park,
\emph{Monoidal categories associated with strata of flag manifolds}, Adv. Math. {\bf328} (2018),  959--1009.

\bibitem{Local} M. Kashiwara, M. Kim, S-j. Oh and E. Park,
 \emph{Localizations for quiver Hecke algebras}, 
in preparation.


\bibitem{KL09}
M.~Khovanov and A. Lauda, {\emph A diagrammatic approach to
categorification of quantum groups  {I}},
 Represent. Theory \textbf{13} (2009), 309--347.

%\bibitem{KL11} \bysame, \emph{A diagrammatic approach to categorification of   quantum groups {II}}, Trans. Amer. Math. Soc. \textbf{363} (2011),  2685--2700.
%

 \bibitem{Kimu12} Y. Kimura, \emph{Quantum unipotent subgroup and dual canonical basis}, Kyoto J. Math. {\bf 52} (2012), no. 2, 277--331.

%\bibitem{KQ14} Y. Kimura and F. Qin, \emph{Graded quiver varieties, quantum cluster algebras, and dual canonical basis}, Adv. Math. {\bf 262} (2014), 261--312.


%\bibitem{KNS11} A. Kuniba, T. Nakanishi and J. Suzuki, \emph{T-systems and Y-systems in integrable systems}, J. Phys. A \textbf{44}, (2011) 103001, 146 pp.

%\bibitem{Lampe11} P. Lampe, \emph{A quantum cluster algebra of Kronecker type and the dual canonical basis}, Int. Math. Res. Not. (2011), no. 13, 2970--3005.

%\bibitem{Lampe14} \bysame \emph{Quantum cluster algebras of type $A$ and the dual canonical basis}, Proc. London Math. Soc. (3) {\bf108} (2014), 1--43.


%\bibitem{LV11} A. Lauda and M. Vazirani, {\em Crystals from categorified quantum groups},Adv. Math. {\bf 228} (2011), 803-861.

%\bibitem{L03} B. Leclerc,{\em Imaginary vectors in the dual canonical basis of $U_q(\mathfrak{n})$},Transform. Groups {\bf 8}  (2003),  no. 1, 95--104.


\bibitem{LS13} K. Lee and R. Schiffler, \emph{Positivity for cluster algebras}, Ann. of Math. {\bf182} (2015) 73--125.

%\bibitem{Lusz90} G. Lusztig, \emph{Canonical bases arising from quantized enveloping algebra}, J. Amer. Math. Soc. \textbf{3}(2) (1990), 447--498.



%\bibitem{Lusz92} \bysame, \emph{Canonical bases in tensor products}, Proc. Natl. Acad. Sci. USA \textbf{89} (1992), 8177--8179.

%\bibitem{Lus93} \bysame,{\em Introduction to Quantum Groups}, Progress in Mathematics {\bf110}, Birkh\"auser Boston (1993), 341 pp.


%\bibitem{Mc14} P. McNamara,{\em Representations of Khovanov-Lauda-Rouquier algebras III:Symmetric Affine Type}, arXiv:1407.7304\/v1.

%\bibitem{Nak13} H. Nakajima, \emph{Cluster algebras and singular supports of perverse sheaves}, Advances in Representation Theory of Algebras, 211--230, EMS Ser. Congr. Rep., Eur. Math. Soc., Z\"urich, 2013.


% \bibitem{Nak11} H. Nakajima, \emph{Quiver varieties and cluster algebras}, Kyoto J. Math. {\bf 51} (2011), 71--126.

\bibitem{Qin17}
F. Qin,  \emph{Triangular bases in quantum cluster algebras and monoidal categorification conjectures},  Duke Math. {\bf166}  (2017), 2337--2442.

\bibitem{R08}
R.~Rouquier, \emph{2-{K}ac-{M}oody algebras},   arXiv:0812.5023\/v1.

 \bibitem{R11} \bysame, {\em Quiver Hecke algebras and 2-Lie algebras}, Algebra Colloq. {\bf 19} (2012), no. 2, 359--410. 
%arXiv:1112.3619\,v1.

\bibitem{TW16}
P. Tingley and B. Webster,
{\em Mirkovi\'c-Vilonen polytopes and Khovanov-Lauda-Rouquier algebras},
Compos. Math. 152 (2016), no. 8, 1648--1696.

%??
\bibitem{VV09}
M. Varagnolo and E. Vasserot,
 \emph{Canonical bases and KLR algebras},
J. reine angew. Math. \textbf{659} (2011), 67--100.
%
\end{thebibliography}
\end{document}